\documentclass{article}
\usepackage[all,ps,arc]{xy}  
\usepackage{amsmath,amssymb,latexsym}
\usepackage{amsfonts}
\usepackage{amsthm}
\usepackage{hyperref}
\newcommand{\point}{\ensuremath{\xymatrix{A\ar@<+.6ex>[r]^(.5){\alpha}
&B\ar@<+.6ex>[l]^(.5){\beta}}}}
\newcommand{\rg}{\ensuremath{\xymatrix{A\ar@<+1ex>[r]^{\alpha}\ar@<-1ex>[r]_{\gamma}&B\ar[l]|{\beta}}}}

\newcommand{\eq}[1]{\ \raisebox{-.5ex}{\xy (0,2)*{\scriptscriptstyle ({#1})}; (0,0)*{=}\endxy}\ }

\newtheorem{Theorem}{Theorem}[section]
\newtheorem{Lemma}[Theorem]{Lemma}
\newtheorem{Proposition}[Theorem]{Proposition}
\newtheorem{Definition}[Theorem]{Definition}

\newtheorem{Corollary}[Theorem]{Corollary}

\newtheorem{Remark}[Theorem]{Remark}
\newtheorem{Example}[Theorem]{Example}

\newcommand{\cA}{\ensuremath{\mathcal A}}
\newcommand{\cB}{\ensuremath{\mathcal B}}
\newcommand{\cC}{\ensuremath{\mathcal C}}

\newcommand{\cE}{\ensuremath{\mathcal E}}

\newcommand{\cV}{\ensuremath{\mathcal V}}
\newcommand{\cI}{{\mathcal I}}
\newcommand{\cL}{{\mathcal L}}

\newcommand{\cH}{\ensuremath{\mathcal H}}
\newcommand{\cK}{\ensuremath{\mathcal K}}
\newcommand{\cM}{{\mathcal M}}

\newcommand{\cX}{\ensuremath{\mathcal X}}

\newcommand\Ker{\text{Ker}}

\newcommand\Pt{\mathbf{Pt}}

\newcommand\Gpd{\mathbf{Gpd}}
\newcommand\Set{\mathbf{Set}}
\newcommand\SET{\mathbf{SET}}
\newcommand\AssAlg{\mathbf{AssAlg}}
\newcommand\Gp{\mathbf{Gp}}

\newcommand\Ext{\mathbf{Ext}}
\newcommand\XExt{\mathbf{XExt}}

\newcommand\XMod{\ensuremath{\mathbf{XMod}}}
\newcommand\Bimod{\mathbf{Bimod}}
\newcommand\XBimod{\mathbf{XBimod}}
\newcommand\XBiext{\mathbf{XBiext}}
\newcommand\Mod{\mathbf{Mod}}
\newcommand\Cat{\ensuremath{\mathbf{Cat}}}
\newcommand\CAT{\ensuremath{\mathbf{CAT}}}
\newcommand\Fib{\mathbf{Fib}}
\newcommand\opFib{\mathbf{opFib}}
\newcommand\Vect{\mathbf{Vect}}

\newdir{|>}{{}*!/8pt/@{|}*!/3.5pt/:(1,-.2)@^{>}*!/3.5pt/:(1,+.2)@_{>}}
\newdir{ >}{{}*!/-10pt/@{>}}
\newdir{ |>}{!/10pt/{}*!/4.5pt/@{|}*:(1,-.2)@^{>}*:(1,+.2)@_{>}}

\begin{document}

\def \tm{\!\times\!}

\newenvironment{changemargin}[2]{\begin{list}{}{
\setlength{\topsep}{0pt}
\setlength{\leftmargin}{0pt}
\setlength{\rightmargin}{0pt}
\setlength{\listparindent}{\parindent}
\setlength{\itemindent}{\parindent}
\setlength{\parsep}{0pt plus 1pt}
\addtolength{\leftmargin}{#1}\addtolength{\rightmargin}{#2}
}\item}{\end{list}}

\title{Fibered aspects of Yoneda's regular span}

\author{A.\ S.\ Cigoli, S.\ Mantovani, G.\ Metere and E.\ M.\ Vitale}

\maketitle

\begin{abstract}
In this paper we start by pointing out that Yoneda's notion of a regular span $S\colon \cX\to\cA\times\cB$ can be interpreted as a special kind of morphism, that we call \emph{fiberwise opfibration}, in the 2-category $\Fib(\cA)$. We study the relationship between these notions and those of internal opfibration and two-sided fibration. This fibrational point of view makes it possible to interpret Yoneda's Classification Theorem given in his 1960 paper as the result of a canonical factorization, and to extend it to a non-symmetric situation, where the fibration given by the product projection $Pr_0\colon\cA\times\cB\to\cA$ is replaced by any split fibration over $\cA$. This new setting allows us to transfer Yoneda's theory of extensions to the non-additive analog given by crossed extensions for the cases of groups and other algebraic structures.

\textbf{Keywords}: cohomology, crossed extension, fibration, regular span.

\textbf{MSC}: 18A22; 18A32; 18D05; 18D30; 18G50; 08C05; 20J06.
\end{abstract}

\section*{Introduction}

In his pioneering 1960 paper \cite{Y60}, Nobuo Yoneda presents a formal categorical setting in order to formulate the classical theory of $\Ext^n$ functors avoiding the request of having enough projectives. The starting point is to study, in an additive category, the functor sending any exact sequence of length $n$
\[
\xymatrix@C=4ex{
0 \ar[r] & b \ar[r]^j & e_n \ar[r] & e_{n-1} \ar[r] & \cdots & \ar[r] & e_1 \ar[r]^p &a\ar[r] &0
}
\]
to the pair $(a,b)$.

The basic observation is that, thanks to the dual properties of pushouts and pullbacks in Yoneda's additive setting, it is possible to define translations and cotranslations of exact sequences (see Section \ref{sec:Yoneda}). This analysis leads Yoneda to identify a set of formal properties of a functor
\[
S\colon\cX\to\cA\times\cB                                                                                                                                                                                                                                                                                                   \]
in order to get the axioms of what he calls a \emph{regular span}. His key idea is the following: ``The $n$-fold extensions over $a$ with kernel $b$ in an additive category will be considered as some quantity lying between $a$ and $b$, or over the pair $(a,b)$, which we want to classify to get $\Ext^n(a,b)$''\footnote{From the introduction of \cite{Y60}.}.

In fact, his Classification Theorem in \cite[$\S 3.2$]{Y60} follows in a purely formal way from the axioms of regular span, once one considers connected components of the fibers of $S$ (called \emph{similarity classes} in [\emph{loc.\ cit.}]). We can interpret (see Theorem \ref{thm:yoneda}) his result by saying that, with any regular span $S$, it is possible to associate what is nowadays called (see \cite{Str74} and also \cite{Weber}) a \emph{two-sided discrete fibration} $\bar S\colon\bar\cX\to\cA\times\cB$, together with a factorization of $S$ through a functor $Q\colon\cX\to\bar\cX$.

Our aim is to extend this interpretation of Yoneda's theory to a non-additive context, where the paradigmatic example is given by crossed $n$-fold extensions. The latter have been adopted to interpret cocycles in non-abelian cohomology for groups and algebras by many authors (see e.g.\ Mac Lane's Historical Note \cite{ML77}), and more recently by Rodelo in \cite{Rodelo}, giving a normalized version of Bourn cohomology \cite{bourn}.

Moving to the category of groups, for example, the first problem one encounters is that any crossed $n$-fold extension determines an $a$-module $(a,b,\xi)$ instead of a mere pair of objects; as a consequence, the functor $S$ has to be replaced by a functor
\[
P\colon \cX \to \cM
\]
whose codomain (the category of group modules) is no longer a product of two categories. Nevertheless, $\cM$ is still the domain of a split fibration (over the category of groups):
\[
\xymatrix{\cX\ar[dr]_{F}\ar[rr]^P&&\cM\ar[dl]^G\\&\cA}
\]
This suggests to look at the notion of regular span by a fibrational point of view, in order to set up a general scheme including also non-additive cases. This will be achieved by considering on one side the notion of \emph{fiberwise opfibration} (see Definition \ref{def:fiberwise}) and on the other side the one of internal opfibration in the 2-category $\Fib(\cA)$ of fibrations over a fixed category $\cA$. The relationship between the two is completely described in Theorem \ref{thm:opfib_in_Fib_A}, and it turns out that a regular span is nothing but a fiberwise opfibration in $\Fib(\cA)$, with codomain a product projection. This result points out the difference between regular spans and two-sided fibrations, that correspond to internal opfibrations in $\Fib(\cA)$ with codomain a product projection, as proved in \cite{BournPenon}.

It is precisely the notion of fiberwise opfibration in $\Fib(\cA)$ with codomain a split fibration that yields a Classification Theorem \emph{\`a la Yoneda} for the non-symmetric case (Theorem \ref{thm:ClassThm}). Still, we can factor $P$ as
\[
\xymatrix@C=8ex{
\cX\ar[r]^Q\ar[dr]_{F}
&\bar\cX\ar[d]^(.35){\bar F}\ar[r]^{\bar P}
&\cM\ar[dl]^{G}
\\
&\cA
}
\]
where $\bar P$ is a discrete opfibration in $\Fib(\cA)$, which is the non-symmetric analog of the two-sided discrete fibration $\bar S$. Our approach from an internal point of view reveals moreover that both Yoneda's construction of $\bar S$ and our generalization $\bar P$ are actually not \emph{ad hoc} constructions. Indeed, they can be recovered as part of a canonical factorization, namely the internal version in $\Fib(\cA)$ of the so called \emph{comprehensive factorization}, given by initial functors and discrete opfibrations (see Proposition \ref{prop:bar_P_opfib} and Theorem \ref{thm:Q_initial}).

Coming back to crossed $n$-fold extensions, in Sections \ref{sec:groups} and \ref{sec:groups_n} we describe in details how the group case fits in our general scheme. The main difference with the additive case is that in order to obtain translations (i.e.\ the opcartesian liftings in the fibers giving rise to a fiberwise opfibration) we have to replace pushouts by \emph{push forwards}. The internal version of the latter, developed in \cite{pf}, is the main tool to treat crossed $n$-fold extensions in any strongly semi-abelian category, as detailed in Section \ref{sec:Moore}, where we explicitly describe opcartesian liftings as push forwards of crossed extensions. The proof of Lemma \ref{lemma:pf_xext} gives a recipe for performing the push forward of crossed extensions in several categories of algebraic varieties, as (non-unital) associative algebras, Lie algebras and Leibniz algebras among others.

The existence of such opcartesian liftings in the fibers was already stated in \cite{Rodelo} as a normalized version of a more general result of Bourn \cite{Bourn02} concerning the so-called \emph{direction} functor. Proposition \ref{prop:Xext_opfib} explains that these opfibrations in the fibers can be gathered together coherently giving rise to an opfibration in $\Fib(\cC)$.

Section \ref{sec:AssAlg} is devoted to show how the interpretation of Hochschild cohomology of unital associative algebras given in \cite{BauMi02} by means of crossed 
biextensions (simply called crossed extensions in [\emph{loc.\ cit.}]) fits in our general scheme. This case cannot be immediately deduced from the general treatment of Section \ref{sec:Moore}, simply because the category $\AssAlg_1$ of unital associative algebras is not semi-abelian. Nevertheless, we can apply our Theorem \ref{thm:fib_opfib_Moore} to the semi-abelian category $\AssAlg$ of (non-unital) associative algebras and prove in Theorem \ref{thm:fib_opfib_AssAlg} that the fiberwise opfibration over $\AssAlg$ so obtained actually restricts to a fiberwise opfibration over $\AssAlg_1$.

\tableofcontents

\section{Exact sequences and Yoneda's regular spans} \label{sec:Yoneda}

In this section, we recall Yoneda's basic definitions and some of his results. For the sake of simplicity, we shall consider the special case of abelian categories, although in \cite{Y60} results are stated in a slightly more general additive setting.

Given an abelian category  $\cA$, we consider the category $\cX$ of exact sequences:\begin{equation} \label{diag:exact_sequence}
\xymatrix@C=4ex{
x \colon & 0 \ar[r] & b \ar[r]^j & e_n \ar[r] & e_{n-1} \ar[r] & \cdots & \ar[r] & e_1 \ar[r]^p & a \ar[r] & 0
}
\end{equation}
with morphisms $f=(\alpha,\cdots,\beta)\colon x\to x'$ being a $(n+2)$-tuple of arrows of $\cA$ between the terms of the sequences, provided they make the obvious diagram commute:
\begin{equation}\label{diag:exact_sequence_arrws}
\begin{aligned}
\xymatrix@C=4ex{
x\ar@<-.88ex>[d]_f \colon& 0\ar[r] &b\ar[r]^j \ar[d]_{\beta} &e_n\ar[r]\ar[d] &e_{n-1}\ar[r]\ar[d] &\cdots &\ar[r] &e_1\ar[r]^p \ar[d] &a\ar[r]\ar[d]^{\alpha} &0
\\
x'\colon& 0\ar[r] &b'\ar[r]_{j'} &e_n'\ar[r] &e'_{n-1}\ar[r] &\cdots &\ar[r] &e'_1\ar[r]_{p'} &a'\ar[r] &0}
\end{aligned}
\end{equation}

If we fix (the identities over) $a$ and $b$, we can identify the subcategory $\cX_{(a,b)}$ of $\cX$ of sequences starting from $b$ and ending in $a$. Yoneda calls \emph{similarity relation} the equivalence relation on the objects of $\cX_{(a,b)}$ generated by its morphisms. In this way, he obtains as quotient what are currently known as the \emph{connected components} of $\cX_{(a,b)}$. We shall denote this quotient by $\bar{\cX}_{(a,b)}$, and its elements by $\bar{x}, \bar{x}_1, \bar{x}_2\dots$.

Now, given a sequence $x$ as above, by the pushout of $j$ along  an arrow $\beta\colon b\to b'$, it is possible to determine a new sequence $\beta_*(x)$ in $\cX_{(a,b')}$ and a morphism $\hat\beta_x=(1_a,\cdots,\beta)\colon x\to\beta_*(x)$ of exact sequences ending in $a$. This process is called \emph{translation} in [\emph{loc.cit.}]. Dually, by the pullback of $p$ along an arrow $\alpha \colon a'\to a$, it is possible to determine a new sequence $\alpha^*(x)$ in $\cX_{(a',b)}$ and a morphism $\hat\alpha^x=(\alpha,\cdots,1_b)\colon\alpha^*(x)\to x$ of exact sequences ending in $b$. This process is called \emph{cotranslation} in [\emph{loc.cit.}]. We shall denote the arrows $\hat\alpha^x$ and $\hat\beta_x$ by $\hat{\alpha}$ and  $\hat{\beta}$ respectively, when no confusion arises.
As a matter of fact, translations and cotranslations are well defined on the similarity classes. Actually, if we define
\begin{equation}\label{eq:actions}
\beta\bullet \bar{x}=\overline{\beta_*(x)} \qquad \bar{x}\bullet\alpha=\overline{\alpha^*(x)}
\end{equation}
the following equations hold
$$
\beta'\bullet(\beta\bullet \bar{x})=(\beta'\cdot\beta)\bullet \bar{x}\qquad
(\bar{x}\bullet \alpha')\bullet\alpha = \bar{x}\bullet (\alpha'\cdot\alpha)
$$
$$
1_b\bullet\bar{x}=\bar{x},\qquad
(\beta\bullet \bar{x})\bullet \alpha=\beta\bullet (\bar{x}\bullet \alpha)\qquad
\bar{x}\bullet 1_a=\bar{x}
$$
for all similarity classes $\bar x$ and arrows $\alpha$, $\alpha'$, $\beta$ and $\beta'$ in $\cA$ for which compositions make sense.
Hence, translations and cotranslations determine two \emph{compatible categorical actions}, or in other words, a functor
$$
\Ext^n\colon \xymatrix{\cA^{op}\times\cA \ar[r]& \SET}
$$
such that $\Ext^n(a,b)=\bar{\cX}_{(a,b)}$.

In order to have a similar interpretation of the functors $\Ext^n$, within a more general settings of additive categories, Yoneda identifies a set of formal properties that constitute the axioms of what he calls a \emph{regular span}.

We shall recall Yoneda's notion of a regular span soon after recalling the axioms of Grothendieck fibration and opfibration. This will present Yoneda's notion in the language of (op)fibrations, but it is worth observing that the latter are actually special cases of the first.

\medskip
\noindent
Let us consider a functor $F\colon \cE \to \cB$.
\begin{itemize}
\item An arrow $f$ of $\cE$ is $F$-\emph{cartesian}, or cartesian w.r.t.\ $F$, whenever for all $g$ of $\cE$ and $\psi$ of $\cB$ such that $F(g)=F(f)\cdot\psi$ there exists a unique arrow $h$ of $\cE$ such that $f\cdot h=g$ and $F(h)=\psi$.
\item The arrow $f$ is $F$-vertical if $F(f)=1$.
\item An arrow $\varphi$ of $\cB$ has a \emph{cartesian lifting} at the object $x$ of $\cE$ if there exists a cartesian arrow $f$ with codomain $x$ such that $F(f)=\varphi$.
\item The functor $F$ is a \emph{fibration} if it admits enough cartesian liftings. Explicitly, if, for every object $x$ of $\cE$ and every arrow $\varphi$ with codomain $F(x)$ there exists at least one cartesian lifting of $f$ at $x$.
\item The functor $F$ is a  \emph{discrete fibration} if, for every object $x$ of $\cE$ and every arrow $\varphi$ with codomain $F(x)$, there exists exactly one lifting of $f$ at $x$. In this case, the lifting is automatically cartesian.
\item An arrow $f$ of $\cE$ is $F$-\emph{opcartesian}, or opcartesian w.r.t.\ $F$, if $f^{op}$ is cartesian w.r.t.\ $F^{op}$. In this way one gets the dual notions of \emph{opcartesian lifting} and of \emph{(discrete) opfibration}.
\item A \emph{cleavage} for a (op)fibration is a pseudo-functorial choice of (op)cartesian liftings.
\item A fibration is called \emph{split} if it admits a functorial cleavage.
\end{itemize}

Throughout this paper we assume that all (op)fibrations admit a cleavage.

Now we are ready to offer Yoneda's definition in the language of (op)fibrations.
\begin{Definition}[\cite{Y60}]\label{def:reg_span}
A \emph{span} from the category $\cB$ to the category $\cA$ is a functor
$$
\xymatrix{
S\colon \cX\ar[r]&\cA\times \cB
}
$$
or equivalently a pair of functors
$$
\xymatrix{
&\cX\ar[dl]_{S_0}\ar[dr]^{S_1}\\
\cA&&\cB
}
$$
where $S_0=Pr_0\cdot S$ and $S_1=Pr_1\cdot S$.
The span $S$ is termed \emph{regular} if
\begin{itemize}
\item[(R0)] $S_0$ has enough $S_1$-vertical $S_0$-cartesian liftings;
\item[(R1)] $S_1$ has enough $S_0$-vertical $S_1$-opcartesian liftings.
\end{itemize}
\end{Definition}
\noindent In other words, the pair of functors $(S_0,S_1)$ is a regular span if
\begin{itemize}
\item $S_0$ is  a fibration such that for every object $x'$ of $\cX$, and every arrow $\xymatrix{a\ar[r]^-{\alpha}&S_0(x')}$ in $\cA$, there exists at least one $S_0$-cartesian arrow $\xymatrix{x\ar[r]^-{\xi}&x'}$ such that $S_0(\xi)=\alpha$ and $S_1(\xi)=1_{S_1(x')}$;
\item $S_1$ is an opfibration such that for every object $x$ of $\cX$, and every arrow $\xymatrix{S_1(x)\ar[r]^-{\beta}&b}$ in $\cB$, there exists at least one $S_1$-opcartesian arrow $\xymatrix{x\ar[r]^-{\xi}&x'}$ such that $S_1(\xi)=\beta$ and $S_0(\xi)=1_{S_0(x)}$.
\end{itemize}
The notion of (regular) span is clearly a self dual notion, so that often results on $S_1$ can be obtained by dualization of the corresponding results for $S_0$. From now on, if no confusion is likely to occur, we will write $i$-vertical and $i$-(op)cartesian for $S_i$-vertical and $S_i$-(op)cartesian respectively.

\begin{Remark}\label{lm:all_vertical}{\em
If $S=(S_0,S_1)$ is a regular span as above, then:
\begin{itemize}
\item All $0$-cartesian liftings of a given arrow $\alpha$ at a chosen object $x$ are $0$-vertically isomorphic to a $1$-vertical $0$-cartesian lifting of $\alpha$ at $x$.
\item All $1$-opcartesian liftings of a given arrow $\beta$ at a chosen object $x$ are $1$-vertically isomorphic to a $0$-vertical $1$-opcartesian lifting of $\beta$ at $x$.
\end{itemize}
}\end{Remark}
\begin{proof}
Obvious, since (op)cartesian liftings are unique, up to a vertical isomorphism.
\end{proof}

Given a regular span $S$ as in Definition \ref{def:reg_span}, for sake of consistency with the example of $n$-fold extensions, we denote by $\cX_{(a,b)}$ its fiber over the object $(a,b)$ of $\cA\times \cB$, and by $\bar{\cX}_{(a,b)}$ the connected components of the fibers.

The following statement is a relevant motivation to formalize the notion of regular span.

\begin{Theorem}[\cite{Y60}, Classification Theorem]\label{thm:yoneda}
Let $S\colon \cX \to \cA\times \cB$ be a regular span as above. Then the assignment
$$
(a,b)\mapsto \bar{\cX}_{(a,b)}
$$
extends to a functor
$$
\Sigma\colon\cA^{op}\times \cB\to\SET\,.
$$
\end{Theorem}
Recall that (when it factors through $\Set$) $\Sigma$ gives rise to what is called a \emph{profunctor} or a \emph{distributor} (see \cite{prof}); the informed reader may already have recognized the definition of a profunctor from the properties that the actions given in (\ref{eq:actions}) must satisfy. On the other hand, Theorem \ref{thm:yoneda} raises a question. We know that the fibers $\cX_{(a,b)}$ gather coherently together in  the category $\cX$. Similarly, also the $\bar{\cX}_{(a,b)}$ can be gathered together in a category $\bar\cX$, which is given explicitly by using the description of $\Sigma$ as a two-sided discrete fibration $\bar S$ (see for example \cite{Weber}). What is not clear at this stage is the relationship between the two categories $\cX$ and $\bar\cX$. Since this will be one of the main points of our investigation on the more general not additive situation, we shall briefly describe the phenomenon in the abelian setting.

We start from the category $\bar{\cX}$. The objects of $\bar{\cX}$ are connected components of the fibers of $S$, i.e.\ equivalence classes, denoted by $\bar{x}$, under the following equivalence relation: $x\sim x'$ if and only if there exists a zig-zag of $S$-vertical arrows connecting $x$ to $x'$. An arrow $\bar x_1\to \bar x_2$ is specified by a pair $(\alpha,\beta)$ of arrows of $\cA$ and $\cB$  such that $dom(\alpha,\beta)=S(x_1)$, $cod(\alpha,\beta)=S(x_2)$, and $\beta\bullet\bar x_1=\bar x_2 \bullet\alpha$.

As usual in the case of profunctors, we can interpret an object $\bar x$ as a \emph{virtual} arrow from $S_0(x)$ to $S_1(x)$. Accordingly, the arrow $(\alpha,\beta)$ described above, can be depicted as a (virtually) commutative diagram:
$$
\xymatrix{
a_1 \ar[d]_{\alpha} \ar@{-->}[r]^{\bar x_1} & b_1 \ar[d]^{\beta}
\\a_2\ar@{-->}[r]_{\bar x_2}&b_2
}
$$
Composition and identities are inherited from the categorical structure of $\cA\times\cB$.
Back to the regular span we started with, the category $\bar\cX$ we have just described appears in a factorization $\bar S\cdot Q$ of $S$:
$$
\xymatrix{\cX\ar@/^4ex/[rr]^-{S}\ar[r]_-{Q}&\bar\cX\ar[r]_-{\bar S}&\cA\times \cB}
$$
This fact was not explicitly recorded in \cite{Y60}; nevertheless, all the necessary ingredients are already present in [\emph{loc.cit}]. The functor $Q$ and the functor $\bar S$ are defined by letting
$$
Q(\xi\colon x_1\to x_2)= S(\xi)\colon \bar x_1\to \bar x_2\,,
$$
and
$$
\bar S ((\alpha,\beta) \colon \bar x_1\to \bar x_2)=(\alpha,\beta)\,.
$$
Clearly the fibers of $\bar S$ are precisely the collections $\bar{\cX}_{(a,b)}$ of connected components of the fibers of $S$.

\smallskip
In conclusion, we can state the main theme of the present work: namely, to give a conceptual understanding of the factorization of $S$, and to extend it to several non abelian settings where the product $\cA\times \cB$ is replaced by a suitable category $\cM$.
This will be achieved by considering properties as \emph{relative to a given base category}. This fibrational point of view is developed in the next section.

\section{Variations on the notion of opfibration over a fixed base}

In this section we shall consider two different instances of the notion of opfibration for a functor $P\colon(\cX,F)\to(\cM,G)$ over a fixed category $\cA$. 
The first concerns the functor $P$ as a 1-cell in the 2-category $\CAT/\cA$, where objects are functors $F$ and $G$ as above, 1-cells are functors $P$ such that $G\cdot P=F$ and a 2-cell $\kappa\colon P\Rightarrow Q\colon F\to G$ is a natural transformation $\kappa\colon P\Rightarrow  Q$ with $G$-vertical components, i.e.\ such that $G\cdot\kappa=id_F$.
\begin{equation}\label{diag:triangle_twocell}
\begin{aligned}
\xymatrix{
\cX \ar@/^2ex/[rr]^{P}_{}="1" \ar@/_2ex/[rr]_{Q}^{}="2"\ar@{=>}"1";"2"^{\kappa} \ar[dr]_{F} & & \cM \ar[dl]^{G} \\
& \cA
}
\end{aligned}
\end{equation}
The second concerns $P$ as a morphism in the 2-category $\Fib(\cA)$, namely the 2-full sub-2-category of $\CAT/\cA$ with objects the fibrations and 1-cells the cartesian functors, i.e.\ those functors over $\cA$ that preserve cartesian arrows.
Occasionally, we shall consider the 2-category $\opFib(\cA)$ of opfibrations over $\cA$.

\smallskip
The fibrational viewpoint suggests the following definition.

\begin{Definition}\label{def:fiberwise}
Let $\cK(\cA)$ be either $\CAT/\cA$, $\Fib(\cA)$ or $\opFib(\cA)$.
\begin{itemize}
\item We say that a morphism $P$ of $\cK(\cA)$ is a \emph{fiberwise (discrete) (op)fibration}  if, for every object $a$ of $\cA$, the restriction  to fibers $P_a\colon \cX_a\to\cM_a$ is a (discrete) (op)fibration.
\item We say that a morphism $P$ of $\cK(\cA)$ is a \emph{fiberwise (op)fibration with globally (op)cartesian liftings} if $P$ is a fiberwise (op)fibration such that (op)cartesian liftings in the fibers are actually (op)cartesian with respect to all the arrows of $\cX$. 
\end{itemize}
\end{Definition}

The reason why we should be interested in fiberwise (op)fibrations becomes evident after the following proposition.

\begin{Proposition}\label{prop:reg_span_and_fibopfib}
The following statements are equivalent:
\begin{itemize}
\item[(i)] The span
$$
\xymatrix@!=4ex{\cX\ar[rr]^-S&&\cA\times\cB}
$$
is regular;

\item[(ii)]  $S$  is a fiberwise opfibration in  $\Fib(\cA)$:
$$
\xymatrix@!=4ex{
\cX\ar[dr]_{S_0}\ar[rr]^S&&\cA\times \cB \ar[dl]^{Pr_0}\\ & \cA
}
$$
\item[(iii)] $S$  is a fiberwise fibration in  $\opFib(\cB)$;
$$
\xymatrix@!=4ex{
\cX\ar[dr]_{S_1}\ar[rr]^S&&\cA\times \cB \ar[dl]^{Pr_1}\\ & \cB
}
$$
\end{itemize}
\end{Proposition}

\begin{proof}
We need to prove only the equivalence between (i) and (ii), since (iii) follows by duality. However, the equivalence between (i) and (ii) is a consequence of  Proposition \ref{prop:opfib_eq_globallyopfib}.
\end{proof}

In order to understand the notion of regular span and its generalization to fiberwise opfibration in contexts, it is necessary to compare these definitions with some internal notion in a 2-category $\cK$. As a matter of fact, if $\cK$ is either $\CAT/\cA$ or $\Fib(\cA)$, then the internal notion of opfibration is related with the ordinary definition of opfibration in $\CAT$. Indeed, $\CAT/\cA$ and $\Fib(\cA)$ are representable 2-categories in the sense of \cite{Str74}; therefore, 2-pullbacks and comma objects are computed by means of 2-pullbacks and comma objects in $\CAT$.  Gray in \cite{Gray66} characterizes Grothendieck opfibrations  by a condition on comma categories, which he calls Chevalley criterion. Actually, this criterion provides a characterization of representable internal opfibrations in a finitely 2-complete 2-category (see \cite{CC}).

\begin{Proposition} \label{prop:Chevalley}
A morphism $p\colon E \to B$ in a 2-category with comma objects is
\begin{itemize}
\item a representable opfibration if and only if the canonical arrow $r$ in the diagram below has a left adjoint with the unit being an identity;
\begin{equation} \label{fib2cat}
\begin{aligned}
\xymatrix{
E \downarrow E \ar@/_1ex/[ddr]_{d_0} \ar@/^1ex/[drr]^{pd_1} \ar[dr]_r \\
& p \downarrow B \ar[r]_{p_1} \ar[d]^{p_0} & B \ar[d]^{1_B} \\
& \ar@{}[ur]|(.3){}="1"\ar@{}[ur]|(.7){}="2"\ar@{=>}"1";"2"_\lambda
E \ar[r]_p & B
}
\end{aligned}
\end{equation}
\item a representable \emph{discrete} opfibration if moreover the counit of the adjunction is an identity. In this case, the comparison $r$ is in fact an isomorphism of categories.
\end{itemize}
\end{Proposition}

We shall call representable internal opfibrations simply \emph{opfibrations}.

\subsection{Opfibrations in $\CAT/\cA$}

The result in the next proposition is already stated (for fibrations) in Example 2.10 of \cite{Weber}, where the property of liftings being globally cartesian is not explicitly highlighted. Actually this is exactly the difference between fiberwise (op)fibrations and (op)fibrations in $\CAT/\cA$, as the following simple counterexample shows.

\begin{Example}
Let us consider the following picture, representing a pair of functors between finite categories, where there are no additional arrows but the identities
\[
\left\{
\begin{array}{ll}
P(\xi)=\mu & P(\xi')=\mu\cdot\mu' \\
G(\mu)=1 & G(\mu')=\alpha
\end{array}
\right.
\]
\[
\xymatrix@C=12ex{
x_0 \ar[drr]^{\xi'}
\\
& x_1 \ar[r]_{\xi} & x_1' & \cX \ar[d]^{P}
\\
m_0\ar[r]^{\mu'}&m_1\ar[r]^{\mu}&m_1'& \cM \ar[d]^{G} \\
a_0\ar[r]^{\alpha}&a_1\ar@{=}[r]&a_1& \cA
}
\]
It is straightforward to check that $P\colon(\cX,GP)\to(\cM,G)$ is a fiberwise fibration, but it is not a fibration in $\CAT/\cA$, since the lifting $\xi$ of $\mu$ is not globally cartesian.
\end{Example}

\begin{Proposition} \label{prop:opfib_in_Cat_A}
A morphism $P\colon (\cX,F)\to (\cM,G)$ in $\CAT/\cA$
\begin{equation}\label{diag:general_triangle}
\begin{aligned}
\xymatrix{\cX\ar[dr]_{F}\ar[rr]^P&&\cM\ar[dl]^G\\&\cA}
\end{aligned}
\end{equation}
is an opfibration in the 2-category $\CAT/\cA$ if an only if it is a fiberwise opfibration with globally opcartesian liftings.
\end{Proposition}

\begin{proof}
Suppose $P$ is an opfibration in $\CAT/\cA$, i.e.\ the arrow $R$ in the diagram below (which is the analog of (\ref{fib2cat}) in $\CAT/\cA$) has a left adjoint with unit an identity.
\begin{equation}\label{diag:def_R}
\begin{aligned}
\xymatrix{
(\cX,F) \downarrow (\cX,F) \ar@/_1ex/[ddr]_{D_0} \ar@/^1ex/[drr]^{PD_1} \ar[dr]_R \\
& P \downarrow (\cM,G) \ar[r]_{P_1} \ar[d]^{P_0} & (\cM,G) \ar[d]^{1_{(\cM,G)}} \\
& \ar@{}[ur]|(.3){}="1"\ar@{}[ur]|(.7){}="2"\ar@{=>}"1";"2"_\lambda
(\cX,F) \ar[r]_P & (\cM,G)
}
\end{aligned}
\end{equation}
Although all constructions are standard, it is worthwhile providing an explicit description of the objects and morphisms involved.

The comma object $P \downarrow (\cM,G)$ can be described as a pair $(\cK,K)$ as follows:
\begin{itemize}
\item an object of $\cK$ is a triple $(x, \beta, m)$, with $x$ in $\cX$, $m$ in $\cM$ and $\beta\colon Px\to m$ an arrow in $\cM$ such that $G(\beta)=1_{Fx=Gm}$
$$
\xymatrix@C=2ex{
(\,x, &Px\ar[rr]^-{\beta}&&m,& m\,)
}
$$
\item a morphism between $(x, \beta, m)$ and $(x', \beta', m')$
is a pair of arrows $(\xi, \mu)$ in $\cX$ and $\cM$ respectively, making the square below commute:
\begin{equation}\label{diag:morphism_H}
\left(
\begin{aligned}
\xymatrix@C=2ex{
x \ar[d]_{\xi}, & Px \ar[rr]^-{\beta} \ar[d]_{P\xi} & & m, \ar[d]^{\mu} & m \ar[d]^{\mu} \\
x', & Px' \ar[rr]_-{\beta'} & & m', & m'
}
\end{aligned}
\right)
\end{equation}
\item the functor $K$ is defined by the assignment on morphisms
$$
K\colon (\xi,\mu)\mapsto F\xi=G\mu\,,
$$
\item the functors $P_0$, $P_1$ and the vertical natural transformation $\lambda$ are defined in the obvious way. For an object $(x,\beta, m)$ one has: $P_0(x,\beta, m)=x$, $P_1(x,\beta, m)=m$ and $\lambda_{(x,\beta, m)}=\beta$.
\end{itemize}

The comma object $(\cX,F) \downarrow (\cX,F) $ can be described similarly as a pair $(\cH,H)$, where:
\begin{itemize}
\item an object of $\cH$ is a triple $(x_0, \nu, x_1)$, with  $\nu\colon x_0\to x_1$ an arrow in $\cX$ such that $F(\nu)=1_{Fx_0=Fx_1}$
$$
\xymatrix@C=2ex{
(\,x_0, & x_0 \ar[rr]^-{\nu} & & x_1,& x_1\,)
}
$$
\item a morphism between $(x_0, \nu, x_1)$ and $(x'_0, \nu', x'_1)$
is a pair of arrows $(\xi_0, \xi_1)$ in $\cX$, making the square below commute:
\[
\left(
\begin{aligned}
\xymatrix@C=2ex{
x_0 \ar[d]_{\xi_0}, & x_0 \ar[rr]^-{\nu} \ar[d]_{\xi_0} & & x_1, \ar[d]^{\xi_1} & x_1 \ar[d]^{\xi_1} \\
x_0', & x_0' \ar[rr]_-{\nu'} & & x_1', & x_1'
}
\end{aligned}
\right)
\]
\item the functor $H$ is defined by extending to morphisms the assignment
$$
H\colon (x_0, \nu, x_1) \mapsto Fx_0=Fx_1 \,,
$$
\item finally we let $D_0(x_0,\nu,x_1)=x_0$, $D_1(x_0,\nu,x_1)=x_1$.
\end{itemize}

The comparison $R$ is defined by extending to morphisms the assignment
$$
R\colon (x_0,\nu,x_1) \mapsto (x_0,P\nu,Px_1) \,.
$$
Let us call $L$ a left adjoint to $R$ in $\CAT/\cA$, $\eta$ and $\epsilon$ the (vertical) unit and counit of the adjunction. By definition, these data have to satisfy the triangle identities
$$
R\epsilon\cdot\eta R=id_{R}\,, \qquad\epsilon L\cdot L\eta = id_{L}\,.
$$
As a consequence, requiring that $\eta$ is an identity, implies that also $R\epsilon$ and $\epsilon L$ are.

Now we shall use the fact that $L$ is left adjoint to $R$ in order to produce opcartesian liftings in the fibers. To this end, let us consider an object $x$ of $\cX$, and an arrow $\beta\colon Px\to m$ in $\cM$ such that $G(\beta)=1_{Fx}$.
These data can be interpreted as an object $(x,\beta,m)$ of $\cK$, so that one can compute $L(x,\beta,m)$.
On the other hand, since the unit of the adjunction is the identity, we know that $RL(x,\beta,m)=(x,\beta,m)$, so that we can legitimately write
$$
L(x,\beta,m)=(x,\hat\beta,\beta_*x)
$$
where $\hat\beta$ is a lifting of $\beta$ at $x$, whose codomain is denoted by $\beta_*x$.
Soon we shall prove that such a lifting is globally opcartesian, but first we have to focus on some special liftings, namely of those objects of $\cK$ that have an underlying identity arrow.

Let us consider the object $(x, 1_{Px}, Px)$, and denote  $L(x, 1_{Px}, Px)=(x, \iota, y)$, where $\iota \colon x\to y$ is the lifting of $1_{Px}$ determined by $L$.
We claim that $\iota$ is an isomorphism. In order to prove this assertion, first we take the component of $\epsilon$ at $(x,1_x,x)$. This must be an arrow
$$
\epsilon_{(x,1_x,x)}=(1_x,\omega)\colon (x,\iota,y)\to (x,1_x, x)\,,
$$
so that $\omega\cdot\iota=1_x$.
Then we take the component of $\epsilon$ at $(x,\iota,y)$:
$$
\epsilon_{(x,\iota,y)}=\epsilon_{L(x,1_{Px},Px)}=(1_x,1_y)\colon (x,\iota,y)\to (x,\iota, y)\,.
$$
Finally, since $LR(1_x,\iota)=L(1_x,1_{Px})=(1_x,1_y)$ we can determine the naturality square of $(1_x,\iota)$
$$
\xymatrix{
(x,\iota,y)\ar[r]^{(1_x,\omega)}\ar[d]_{LR(1_x,\iota)}
&(x,1_x,x)\ar[d]^{(1_x,\iota)}
\\
(x,\iota,y)\ar[r]_{1_x,1_y} &(x,\iota,y)}
$$
whose commutativity is equivalent to the equation $\iota\cdot \omega=1_y$. Therefore, $\iota$ is an isomorphism as announced.

Now we return to the problem of showing that the lifting $\hat\beta$ is globally opcartesian.
To this end, let us consider arrows $\xi$ of $\cX$ and $\mu'$ of $\cM$ such that $P(\xi)=\mu'\cdot\beta$ (and therefore $F(\xi)=G(\mu')$). We need to find a unique arrow $\xi'$ such that $\xi'\hat\beta=\xi$ and $P(\xi')=\mu'$.
\begin{equation}\label{diag:opcartesian}
\begin{aligned}
\xymatrix@C=12ex{
&&x'&\cX
\\
x\ar@/^2ex/[urr]^{\xi}\ar[r]_{\hat\beta}
&\beta_*x\ar@{-->}[ur]_{\xi'}
&&
\\
Px\ar[r]_{\beta}&m\ar[r]_{\mu'}&m'& \cM
}
\end{aligned}
\end{equation}
This data yield the following arrow of $\cK$
\[
\left(
\begin{aligned}
\xymatrix@C=2ex{
x \ar[d]_{\xi}, & Px \ar[rr]^-{\beta} \ar[d]_{P\xi} & & m, \ar[d]^{\mu'} & m \ar[d]^{\mu'} \\
x', & Px' \ar[rr]_-{1} & & m', & m'
}
\end{aligned}
\right)
\]
Applying the left adjoint $L$ to this arrow, we obtain an arrow of $\cH$
\[
\left(
\begin{aligned}
\xymatrix@C=2ex{
x \ar[d]_{\xi}, & x \ar[rr]^-{\hat\beta} \ar[d]_{\xi} & & \beta_*x, \ar[d]^{\xi_0} & \beta_*x \ar[d]^{\xi_0} \\
x', & x' \ar[rr]_-{\iota'} & & y', & y'
}
\end{aligned}
\right)
\]
Now it suffices to put $\xi'=\omega'\cdot \xi_0$, where $\omega'$ is the inverse of the isomorphism $\iota'$. Therefore:
$$
\xi'\cdot \hat\beta =\omega'\cdot \xi_0\cdot \hat\beta=\omega'\cdot\iota'\cdot\xi=\xi\,,
$$
and
$$
P\xi'=P(\omega'\cdot \xi_0)=P(\xi_0)=\mu'\,.
$$
The comparison $\xi'$ is indeed unique. For, let $\varphi$ be another arrow such that $\varphi\cdot\hat\beta=\xi$ and $P(\varphi) =\mu'$. Then, the two pairs $(\xi,\xi')$ and $(\xi,\varphi)$  determine two parallel arrows
$$
\xymatrix@C=10ex{L(x,\beta,m)=(x,\hat\beta,\beta_*x)\ar@/^4ex/[r]^{(\xi,\xi')}\ar@/_4ex/[r]_{(\xi,\varphi)}&(x',1_{x'},x')
}
$$
Since these arrows are sent by the adjoint correspondence onto the same arrow
$$
\xymatrix@C=10ex{(x,\beta,m)\ar[r]^-{(\xi,\mu')}&R(x',1_{x'},x')=(x',1_{m'},m')}
$$
they are equal, i.e.\ $\varphi=\xi'$.

\medskip
Vice versa, let us suppose that $P$ is a fiberwise opfibration with globally opcartesian liftings. We are to define a functor $L$ such that $(L,R,\eta,\epsilon)$ is an adjunction over $\cA$, such that $\eta$ is the identity transformation.

After choosing a collection of cleavages $\hat{(\ )}$, one for each fiber
\begin{itemize}
\item on objects, we let
$$
L(x,\beta,m)=(x,\hat\beta,\beta_*x)\,;
$$
\item in order to define $L$ on $(\xi,\mu)$ (see diagram (\ref{diag:morphism_H})), let us consider the following diagram:
$$
\xymatrix{
x \ar[r]^-{\hat\beta} \ar[d]_{\xi}
& \beta_*x \ar@{-->}[d]^{\varphi} \\
x' \ar[r]_-{\hat\beta'} & \beta'_*x'
}
$$
Since $\hat\beta$ is a globally opcartesian lifting of $\beta$, there exists a unique $\varphi$ such that $\varphi\cdot\hat\beta= \hat{\beta'}\cdot \xi$, and $P\varphi=\mu$.
Let us conclude that, since the opcartesian liftings have been chosen, then this assignment is functorial.
\end{itemize}
The functor $L$ we have just defined lives indeed over $\cA$, since $H\cdot L=K$, and
it is a left adjoint to $R$ in $\CAT/\cA$.

Since $RL=id$, we can choose the identity of $RL$ as unit of the adjunction. For what concerns the counit, for any object $(x_0,\nu,x_1)$ of $\cH$, since $F\nu=1$, there is a factorization $\nu=\omega\cdot \widehat{P\nu}$. Hence we can define the counit as $\epsilon_{(x_0,\nu,x_1)}=(1_{x_0},\omega)$:
\[
\xymatrix{
x_0\ar[d]_{1_{x_0}}\ar[r]^-{\widehat{P\nu}} & (P\nu)_*x_0 \ar@{-->}[d]^{\omega} \\
x_0\ar[r]_{\nu} & x_1
}
\]
This assignment is obviously natural. Moreover,
$$
R(\epsilon_{(x_0,\nu,x_1)}) = R(1_{x_0},\omega)=(1_{x_0},1_{Px_1})=1_{(x_0,P\nu,Px_1)}
$$
so that the equality $R\epsilon\cdot\eta R=R\epsilon=id_R$ holds.
On the other hand,
$$
\epsilon_{L(x,\beta,m)}=\epsilon_{(x,\hat\beta,\beta_*x)}=id
$$
so that $\epsilon L\cdot L\eta=\epsilon L=id$ as desired.
\end{proof}

\begin{Corollary} \label{cor:disc_opfib_in_Cat_A}
A morphism $P\colon (\cX,F)\to (\cM,G)$ in $\CAT/\cA$ as in diagram (\ref{diag:general_triangle})
is a discrete opfibration in the 2-category $\CAT/\cA$ if an only if it is a fiberwise discrete opfibration with globally opcartesian liftings.
\end{Corollary}

\begin{proof}
If $P$ is a discrete opfibration in  $\CAT/\cA$, then the functor $R$ is an isomorphism over $\cA$. In particular, its inverse $L$ is a left adjoint and the counit is an identity. Therefore, the opcartesian liftings in the fibers are uniquely determined, and they are globally opcartesian by Proposition \ref{prop:opfib_in_Cat_A}. Vice versa, if $P$ is a fiberwise discrete opfibration with globally opcartesian liftings, it determines a left adjoint to $R$ in $\CAT/\cA$. However in this case the counit is an identity, since it is given by the vertical comparison between two identical liftings, and therefore $R$ is an isomorphism over $\cA$.
\end{proof}

\subsection{Opfibrations in $\Fib(\cA)$} \label{sec:opfib_FibA}

Next proposition establishes that the two notions of \emph{fiberwise opfibration} and
\emph{fiberwise opfibration with globally opcartesian liftings} coincide in $\Fib(\cA)$.

\begin{Proposition}\label{prop:opfib_eq_globallyopfib}
A morphism $P\colon (\cX,F)\to (\cM,G)$ in $\Fib(\cA)$ as in diagram (\ref{diag:general_triangle}) is a fiberwise opfibration if an only if it is a fiberwise opfibration with globally opcartesian liftings.
\end{Proposition}

\begin{proof}
Let us suppose $P$ is just a fiberwise opfibration, and consider the situation described in diagram (\ref{diag:opcartesian}), where $P(\hat\beta)=\beta$ and $G(\beta)=1_{Fx}$, with $\hat\beta$ being an opcartesian lifting of $\beta$ in the fiber over $Fx$. Let us consider the following diagram
$$
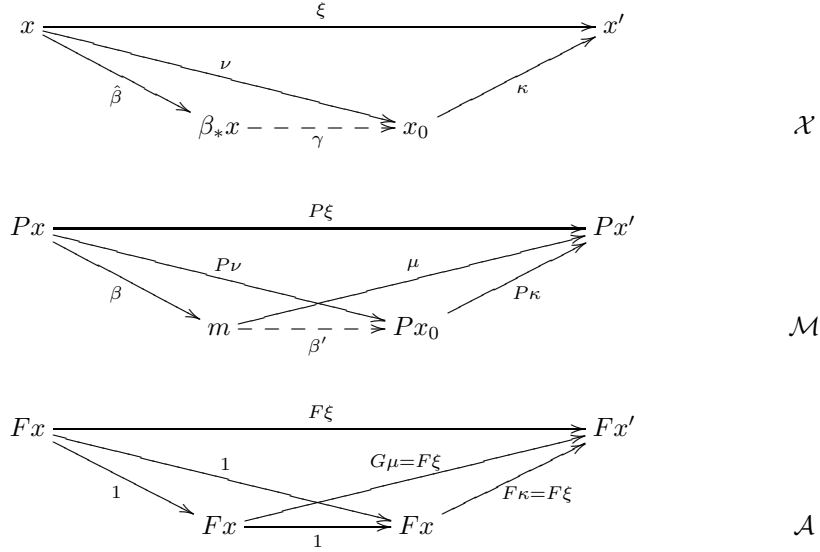
\xymatrix@C=12ex{
x\ar[dr]_{\hat\beta}\ar[drr]^\nu\ar[rrr]^{\xi}&&&x'
\\
&\beta_*x\ar@{-->}[r]_{\gamma} & x_0 \ar[ur]_{\kappa} & & \cX
\\
Px\ar[dr]_{\beta}\ar[drr]^{P\nu}\ar[rrr]^{P\xi}&&&Px'
\\
&m\ar@{-->}[r]_{\beta'}\ar[urr]^{\mu}&Px_0\ar[ur]_{P\kappa}&&\cM
\\
Fx\ar[dr]_{1}\ar[drr]^{1}\ar[rrr]^{F\xi}&&&Fx'
\\
&Fx\ar[r]_{1}\ar[urr]^{G\mu=F\xi}&Fx\ar[ur]_{F\kappa=F\xi}&&\cA
}
$$
\begin{itemize}
\item[(i)] since $F$ is a fibration, we can factor $\xi=\kappa\cdot\nu$, with $\kappa$ $F$-cartesian over $F\xi$ and $\nu$ $F$-vertical;
\item[(ii)] since $P$ preserves cartesian arrows, $P\kappa$ is $G$-cartesian;
\item[(iii)] since $G\mu=G(\mu\cdot\beta)=G(P\kappa\cdot P\nu)=G(P\kappa)$, there exists a unique $\beta'$ such that $P\kappa\cdot \beta'=\mu$ and $G\beta'=1_{Fx}$;
\item[(iv)] moreover, $\beta'\cdot \beta=P\nu$, since $G(\beta'\cdot \beta)=1_{Fx}=F(\nu)=G(P\nu)$ and $P\kappa$ is cartesian; 
\item[(v)] since $\hat\beta$ is an opcartesian lifting of $\beta$ in the fiber over $Fx$, there exists a unique $\gamma$ such that $\gamma\cdot\hat\beta=\nu$ and $P\gamma=\beta'$;
\item[(vi)] finally, $P(\kappa\cdot\gamma)=P\kappa\cdot\beta'=\mu$.
\end{itemize}
From the steps described above, we can define an arrow $\xi'=\kappa\cdot\gamma\colon \beta_*x\to x'$, such that $\xi'\cdot\hat\beta=\xi$ and $P(\xi')=\mu$.
The fact that this arrow is unique is clear from the proof.
\end{proof}


As we shall shortly explain, (although necessary) being a fiberwise opfibration in $\Fib(\cA)$ is not enough in order be an opfibration in the 2-category $\Fib(\cA)$.
What is missing is precisely the condition that we are going to introduce.

Let us be given a fiberwise opfibration $P\colon (\cX,F)\to (\cM,G)$ in $\Fib(\cA)$ as in diagram (\ref{diag:general_triangle}), an object $x$ of $\cX$ and a pair of arrows $\alpha\colon a\to Fx$ in $\cA$ and $\beta\colon Px\to m$ in $\cM$ such that $G\beta=1_{Fx}$. Consider the diagram below
\begin{equation} \label{diag:C}
\xymatrix{
& (\alpha^*\beta)_*\alpha^*x \ar@{-->}[d]^{\omega} \\
\alpha^*x \ar[ur]^{\widehat{P\alpha^*\hat\beta}} \ar[d]_1 & \alpha^*\beta_*x \ar[rr]^{\hat\alpha} & & \beta_*x & \cX \\
\alpha^*x \ar[ur]_(.65){\alpha^*\hat\beta} \ar[rr]_{\hat\alpha} & & x \ar[ur]_{\hat\beta} \\
& P\alpha^*\beta_*x\ar[r]^{\tau'}_{\sim} & \alpha^*m' \ar[r]^{\hat\alpha} & m & \cM \\
P\alpha^*x \ar[ur]^{P\alpha^*\hat\beta} \ar[r]^{\tau}_{\sim} & \alpha^*Px \ar[ur]_(.65){\alpha^*\beta} \ar[r]_{\hat\alpha} & Px \ar[ur]_{\beta}
\\
& a \ar[r]^{1} & a \ar[r]^{\alpha} & Fx & \cA \\
a \ar[ur]^{1} \ar[r]^{1} & a \ar[ur]_(.65){1} \ar[r]_{\alpha} & Fx \ar[ur]_{1}
}
\end{equation}
First take cartesian liftings $\hat\alpha$ of $\alpha$ at $Px$ and at $m$ and call $\alpha^*\beta$ the unique comparison such that $\beta\cdot\hat\alpha=\hat\alpha\cdot\alpha^*\beta$. Then take an opcartesian lifting $\hat\beta$ of $\beta$ at $x$, cartesian liftings of $\alpha$ at $x$ and at $\beta_*x$ and call as before $\alpha^*\hat\beta$ the corresponding comparison. $P\alpha^*\hat\beta$ is isomorphic to $\alpha^*\beta$ through the pair of isomorphisms $(\tau,\tau')$ displayed above. With a little abuse of notation, let us call $(\alpha^*\beta)_*\alpha^*x$ the codomain of the opcartesian lifting of $P\alpha^*\hat\beta$ at $\alpha^*x$. By opcartesianness, there exists a unique $P$-vertical comparison $\omega$ such that $\omega\cdot\widehat{P\alpha^*\hat\beta}=\alpha^*\hat\beta$. Let us observe that, in fact, the pair $(1_{\alpha^*x},\omega)$ can be interpreted as a counit component at $(\alpha^*x,\alpha^*\hat\beta,\alpha^*\beta_*x)$ of a chosen adjunction $L\dashv R$ as in the proof of Proposition \ref{prop:opfib_in_Cat_A}.

We say that $P$ satisfies (C) if the following condition holds:
\begin{itemize}
\item[(C)] For any object $x$ of $\cX$ and any pair of arrows $\alpha\colon a\to Fx$ in $\cA$ and $\beta\colon Px\to m$ in $\cM$ such that $G\beta=1_{Fx}$, the canonical comparison
$$
\omega\colon (\alpha^*\beta)_*\alpha^*x\to\alpha^*\beta_*x
$$
is an isomorphism.
\end{itemize}
Although formally not precise, this notation has two advantages: first, it expresses in a short way a procedure which really only depends on the choice of $\alpha$, $\beta$ and $x$; second, it allows to interpret condition (C) above as an interchange law between the fibrations and opfibrations involved.

\begin{Theorem} \label{thm:opfib_in_Fib_A}
For a  morphism $P\colon (\cX,F)\to (\cM,G)$ in $\Fib(\cA)$ as in diagram (\ref{diag:general_triangle}), the following statements are equivalent:
\begin{itemize}
\item[(I)] $P$ is an opfibration in the 2-category $\Fib(\cA)$;
\item[(II)] $P$ is a fiberwise opfibration in $\Fib(\cA)$ and condition \emph{(C)} holds;
\item[(III)] $P$ is a fiberwise opfibration in $\Fib(\cA)$ with globally opcartesian liftings and condition \emph{(C)} holds.
\end{itemize}
\end{Theorem}
\begin{proof}
That (II) is equivalent to (III) is clear from Proposition \ref{prop:opfib_eq_globallyopfib}.

Let us prove that (I) $\Rightarrow$ (II). We know from the description of comma objects in $\Fib(\cA)$ given in \cite{He99} that (\ref{diag:def_R}) for such a $P$ is a diagram also in $\Fib(\cA)$ (and not only in  $\CAT/\cA$). 
By Proposition \ref{prop:Chevalley}, (I) means that $R$ admits a left adjoint $L$ in $\Fib(\cA)$, with unit $\eta=id$ (such an $L$ realizes a collection of cleavages, one for each fiber). Condition (C) will follow from the fact that $L$ is cartesian.

Indeed, given the triple $\alpha$, $\beta$ and $x$ as in the hypothesis of condition (C), we can construct the cartesian lifting of $\alpha$ at $(x,\beta,m)$ displayed in the left hand side of the diagram below (see also diagram (\ref{diag:C})). Then its image under $L$ must be cartesian in $\cH$ (where $(\cH,H)=(\cX,F) \downarrow (\cX,F)$):
\begin{equation}\label{diag:def_L}
\begin{aligned}
\xymatrix@R=2ex{
\alpha^*x\ar[r]^{\hat\alpha}&x
\\
P\alpha^*x \ar[r]^-{\hat\alpha\cdot\tau} \ar[dd]_{P\alpha^*\hat\beta} & Px \ar[dd]^{\beta}
\\ \\
P\alpha^*\beta_*x \ar[r]_-{\hat\alpha\cdot\tau'}
& m \\
P\alpha^*\beta_*x \ar[r]_-{\hat\alpha\cdot\tau'}
& m
}
\end{aligned}
\qquad\stackrel{L}{\mapsto}\qquad
\begin{aligned}
\xymatrix@R=2ex{
\alpha^*x\ar[r]^{\hat\alpha}&x
\\
\alpha^*x \ar[r]^{\hat\alpha} \ar[dd]_{\widehat{P\alpha^*\hat\beta}} & x \ar[dd]^{\hat\beta}
\\ \\
(\alpha^*\beta)_*\alpha^*x\ar[r]_-{\hat\alpha\cdot\omega}
& \beta_*x
\\
(\alpha^*\beta)_*\alpha^*x\ar[r]_-{\hat\alpha\cdot\omega}
& \beta_*x}
\end{aligned}
\end{equation}
Hence, since projections are cartesian, $\hat\alpha\cdot\omega$ must be a cartesian arrow, and the comparison $\omega$ is an isomorphism:
$$
\xymatrix{
(\alpha^*\beta)_*\alpha^*x\ar[dr]^-{\hat\alpha\cdot\omega} \ar[d]_{\omega}
\\
\alpha^*\beta_*x\ar[r]_{\hat \alpha}
&\beta_*x
}
$$

Now, let us prove that (III) $\Rightarrow$ (I). By Proposition \ref{prop:opfib_in_Cat_A}, $P$ is an opfibration in the 2-category $\CAT/\cA$. Therefore, what remains to be proved is just that $L$ is cartesian. This is a direct consequence of condition (C). Indeed, if we consider the assignment given in diagram (\ref{diag:def_L}), we only have to prove that $\hat\alpha\cdot\omega$ is cartesian in $\cM$, but this is a consequence of the vertical isomorphism provided by condition (C), which compares $\hat\alpha\cdot\omega$ with the cartesian lifting $\hat\alpha$.
\end{proof}

\begin{Corollary} \label{cor:disc_opfib_in_Fib_A}
A  morphism $P\colon (\cX,F)\to (\cM,G)$ in $\Fib(\cA)$ as in diagram (\ref{diag:general_triangle}) is a discrete opfibration in the 2-category $\Fib(\cA)$ if and only if it is a fiberwise discrete opfibration in $\Fib(\cA)$.
\end{Corollary}

\begin{proof}
If $P$ is a discrete opfibration in $\Fib(\cA)$, then by Theorem \ref{thm:opfib_in_Fib_A} it is a fiberwise discrete opfibration in $\Fib(\cA)$. Conversely, if $P$ is fiberwise discrete opfibration in $\Fib(\cA)$, then the $P$-fibers are discrete, hence the canonical arrow $\omega$ of condition (C) is an identity.
\end{proof}

The results presented so far allow us to identify very accurately the notion of Yoneda's regular span. As a consequence, some different notions already present in literature can be described by considering $\cM=\cA\times \cB$, and $G=Pr_0$. As a matter of fact, by Proposition \ref{prop:reg_span_and_fibopfib}, a regular span $S$  precisely determines a fiberwise opfibration in $\Fib(\cA)$ with codomain a product projection. In this case, asking for condition (C) to hold amounts to the following request:
\begin{itemize}
\item For any object $x$ of $\cX$ and any pair of arrows $\alpha\colon a\to S_0x$ in $\cA$ and $\beta\colon S_1x\to b$ in $\cB$, the canonical comparison
$$
\omega\colon \beta_*\alpha^*x\to\alpha^*\beta_*x
$$
is an isomorphism.
\end{itemize}
A regular span satisfying this condition is called a \emph{two-sided fibration} \cite{Str74}. Then, specilizing Theorem \ref{thm:opfib_in_Fib_A} to this case, we recover the following result, originally due to Bourn and Penon.

\begin{Proposition}[\cite{Y60}] \label{prop:BP}
The following statements are equivalent:
\begin{itemize}
\item[(i)] The span
$$
\xymatrix@!=4ex{\cX\ar[rr]^-S&&\cA\times\cB}
$$
is a two-sided fibration;

\item[(ii)]  $S$  is an opfibration in $\Fib(\cA)$:
$$
\xymatrix@!=4ex{
\cX\ar[dr]_{S_0}\ar[rr]^S&&\cA\times \cB \ar[dl]^{Pr_0}\\ & \cA
}
$$
\item[(iii)] $S$  is a fibration in  $\opFib(\cB)$;
$$
\xymatrix@!=4ex{
\cX\ar[dr]_{S_1}\ar[rr]^S&&\cA\times \cB \ar[dl]^{Pr_1}\\ & \cB
}
$$
\end{itemize}
\end{Proposition}

In fact, the main example of regular span in \cite{Y60} given by abelian $n$-fold extensions is actually a two-sided fibration. On the other hand, as we pointed out in Section \ref{sec:Yoneda}, every regular span $S$ determines a two-sided discrete fibration $\bar S$. This means that Yoneda's notion of regular span is not quite a two-sided fibration, but its structure is enough to ensure that the discrete version of condition (C) holds for $\bar S$.

As recalled above, the notion of regular span is auto-dual. This is not only reflected in the symmetry of the definition, but also in the concrete example of $n$-fold extensions, where it relies on the auto-dual categorical structure of the base category:  cotranslations are obtained by pullback, while translations by pushout.
This symmetry is broken when we move to the non-abelian case of crossed $n$-fold extensions, which is studied in details in Section \ref{sec:Examples}. This is the reason why we needed to generalize regular spans to fiberwise opfibrations and two-sided fibrations to opfibrations between fibrations over a fixed base.

\section{The non-symmetric Classification Theorem} \label{sec:barP}

From now on, we shall be interested in a fiberwise opfibration $P\colon (\cX,F)\to (\cM,G)$:
\begin{equation}\label{diag:fib_opfib_2}
\begin{aligned}
\xymatrix{\cX\ar[dr]_{F}\ar[rr]^P&&\cM\ar[dl]^G\\&\cA}
\end{aligned}
\end{equation}
where the fibration $G$ is split. This means that, for $\alpha\colon a_1\to a_2$ and $\alpha'\colon a_2\to a_3$ in $\cA$, we can choose a cleavage in such a way that
$$
\alpha^*\cdot\alpha'^*=(\alpha'\cdot\alpha)^*\colon \cM_{a_3}\to \cM_{a_1}\quad\text{and}\quad (1_{a_1})^*=1_{\cM_{a_1}}
$$
where $(\ )^*$ turns any arrow in its corresponding change of base functor.

A relevant consequence of the notion of fiberwise opfibration is that we can make cartesian $F$-liftings compatible with the chosen cartesian $G$-liftings. This is a consequence of the following lemma.

\begin{Lemma}\label{lm:compatibility_of_P}
Let us consider  a fiberwise opfibration in $\Fib(\cA)$ as in diagram (\ref{diag:fib_opfib_2}), with G a split fibration. Then, for each $\alpha\colon a\to F(x)$, every $P$-cartesian lifting $\hat\alpha^m$ of $\alpha$ at $m$ in $\cM$ and every object $x$ of $\cX$, there exists an $F$-cartesian lifting $\hat\alpha^x$ of $\alpha$ at $x$ such that $P(\hat\alpha^x)=\hat\alpha^m$.
\end{Lemma}

\begin{proof}
Take any $F$-cartesian lifting $\kappa$ of $\alpha$ at $x$. Then, $P(\kappa)$ is a cartesian lifting of $\alpha$ at $m$, and therefore there exists a $G$-vertical isomorphism $\tau$ such that $P(\kappa)=\hat\alpha^m\cdot\tau$. Now, since $P_{a}\colon \cX_a\to \cM_a$ is an opfibration, there exists a $F$-vertical opcartesian lifting of $\tau$ at the domain of $\kappa$. Denote $\hat\tau$ such a lifting, and define $\hat\alpha^x = \kappa\cdot\hat\tau^{-1}$. Of course $\hat\alpha^x$ is cartesian, since it is vertically isomorphic to the cartesian arrow $\kappa$. Moreover, $P(\hat\alpha^x)=P(\kappa\cdot\hat\tau^{-1})=P(\kappa)\cdot P(\hat\tau^{-1})=\hat\alpha^m\cdot\tau\cdot\tau^{-1}=\hat\alpha^m$.
\end{proof}

\medskip
The rest of this section is devoted to the proof of the following result, which is the generalization of Yoneda's Classification Theorem \ref{thm:yoneda}.

\begin{Theorem}[Non-symmetric Classification Theorem] \label{thm:ClassThm}
A fiberwise opfibration $P$ in $\Fib(\cA)$ with codomain a split fibration can be factorized as
\begin{equation}\label{diag:factorization_of_p}
\begin{aligned}
\xymatrix@C=8ex{
\cX\ar[r]^Q\ar[dr]_{F}
&\bar\cX\ar[d]^(.35){\bar F}\ar[r]^{\bar P}
&\cM\ar[dl]^{G}
\\
&\cA
}
\end{aligned}
\end{equation}
where $\bar P$ is a discrete opfibration in $\Fib(\cA)$ and any fiber $\bar\cX_m$ consists of the connected components of $\cX_m$.
\end{Theorem}

We will construct $Q$ as a 2-categorical analog of a quotient of a suitable 2-kernel of $P$, but at the same time as part of the \emph{comprehensive} factorization of $P$ in the 2-category $\Fib(\cA)$.

\subsection{Description of $\bar{\mathcal X}$}\label{sec:bar_X}

The category $\bar\cX$ is defined as follows.
\begin{itemize}
\item Objects of $\bar\cX$ are equivalence classes $\bar x$ of objects of $\cX$ under the following equivalence relation: $x\sim x'$ if and only if there exists a zig-zag of $P$-vertical arrows connecting $x$ to $x'$.
\item A morphism $\mu\colon\bar x_1\to \bar x_2$ is determined by a morphism $\mu\colon Px_1\to Px_2$ of $\cM$, satisfying
$$\overline{\beta_*x_1}=\overline{\alpha^*x_2}$$
where:
\begin{itemize}
\item[(i)] $\alpha=\alpha_\mu=G\mu\,$;
\item[(ii)] $\beta=\beta_\mu$ is the unique morphism of $\cM$ such that $P\hat\alpha^{x_2}\cdot \beta=\mu$, with $\hat\alpha^{x_2}$ a cartesian lifting of $\alpha$ at $x_2$, and $G\beta=1\,$.
\end{itemize}
\end{itemize}
\begin{Remark}\label{rk:canonical_liftings}{\em
Notice that, as a consequence of Lemma \ref{lm:compatibility_of_P}, one can always fix a lifting $\hat\alpha^{Px_2}$ of $\alpha$ at $Px_2$ and then choose $\hat\alpha^{x_2}$ in such a way that $P\hat\alpha^{x_2}=\hat\alpha^{Px_2}$:
$$
\xymatrix@C=12ex{
x_1'\ar[r]^{\hat\alpha^{x_2}}&x_2&\cX
\\
Px_1\ar[dr]^{\mu}\ar@{-->}[d]_{\beta}
&&
\\
m\ar[r]_{P\hat\alpha^{x_2}=\hat\alpha^{Px_2}}
&Px_2&\cM
\\
Fx_1\ar[r]_{\alpha=G\mu}
&Fx_2&\cA
}
$$
We shall assume these choices throughout this section, whenever it will be necessary.  
We will denote both $\hat\alpha^{x_2}$ and $\hat\alpha^{Px_2}$ by $\hat\alpha$ when no confusion is likely to happen.
}\end{Remark}

Before we can describe the composition of morphisms, the next statement is in order.

\begin{Lemma}\label{lm:bar_X_morphs_well_defined}
Morphisms of $\bar\cX$ are well defined.
\end{Lemma}
\begin{proof}
Let us consider a morphism $\mu\colon\bar x_1\to \bar x_2$ as above, and choose $y_i\in \bar x_i$, with $i=1,2$. We want to prove that
$$\overline{\beta_*y_1}=\overline{\alpha^*y_2}$$
We will use the following diagram.
\begin{equation} \label{diag:bar_X_morphs_well_defined}
\begin{aligned}
\xymatrix@C=4ex@R=4ex{
y_1\ar[rr]^{\hat\beta^{y_1}}&&\beta_*y_1
\\
\vdots\ar[u]^-{\omega_0'}&&\vdots\ar[u]_-{\gamma_0'}
\\
z\ar[u]^-{\omega_2'}\ar[d]_-{\omega_1'}\ar[rr]^{\hat\beta^{z}}
&&\ar[u]_-{\gamma_2'}\ar[d]^-{\gamma_1'}
\\
x_1\ar[rr]_-{\hat\beta^{x_1}}
&&\beta_*x_1
&\ar[l]_-{\omega_0}
\cdots&\ar[l]_-{\omega_2}\ar[r]^-{\omega_1}
&\alpha^*x_2\ar[rr]^{\hat\alpha^{x_2}}
&&x_2
\\
&&&&&\ar[u]^-{\gamma_1''}\ar[d]_-{\gamma_2''}\ar[rr]_{\hat\alpha^w}
&&w\ar[u]_-{\omega_1''}\ar[d]^-{\omega_2''}
\\
&&&&&\vdots\ar[d]_-{\gamma_0''}
&&\vdots\ar[d]^-{\omega_0''}
\\
&&&&&\alpha^*y_2\ar[rr]_{\hat\alpha^{y_2}}
&&y_2
}
\end{aligned}
\end{equation}
By hypothesis, we can assume that
\begin{itemize}
\item[(i)]  $\overline{\beta_*x_1}=\overline{\alpha^*x_2}$ is realized by the zig-zag of $\omega_i$'s connecting $\beta_*x_1$ with $\alpha^*x_2$ at the centre of the diagram, with $P\omega_i=1$;
\item[(ii)]  $y_1\in\bar x_1$ is realized by the zig-zag of $\omega_i'$'s connecting $x_1$ with $y_1$ at the left-hand side of the diagram, with $P\omega_i'=1$;
\item[(iii)]  $y_2\in\bar x_2$ is realized by the zig-zag of $\omega_i''$'s connecting $x_2$ with $y_2$ at the right-hand side of the diagram, with $P\omega_i''=1$.
\end{itemize}
For what concerns the left-hand side of the diagram, we can proceed in this way: since $G\beta=1$, we can consider the opcartesian lifting $\hat\beta^{z}$ of $\beta$ at $z$. Hence, since $\hat\beta^{z}$ is opcartesian over $\beta$, there exists a unique $\gamma_1'$ such that $\gamma_1'\cdot\hat\beta^{z}=\hat\beta^{x_1}\cdot \omega_1'$ and $P\gamma_1'=1$. Iterating the process, we obtain a zig-zag of $\gamma_i'$'s connecting $\beta_*x_1$ with $\beta_*y_1$, with $P\gamma_i'=1$, i.e.\ $\overline{\beta_*x_1}=\overline{\beta_*y_1}$.
For what concerns the right-hand side of the diagram, we can operate similarly: we consider the cartesian lifting $\hat\alpha^{w}$ of $\alpha$ at $w$. Hence, since $\hat\alpha^{x_2}$ is cartesian over $\alpha$, there exists a unique $\gamma_1''$ such that $\hat\alpha^{x_2}\cdot \gamma_1''=\omega_1''\cdot \hat\alpha^{w}$ and $F\gamma_1''=1$. Moreover, $P\omega_1''=1$ and, by Remark \ref{rk:canonical_liftings}, $P\hat\alpha^{x_2}=\hat\alpha^{Px_2}=P\hat\alpha^{w}$ is cartesian. Therefore $P\gamma_1''=1$ by uniqueness of the comparison between $P\hat\alpha^{x_2}$ and $P\hat\alpha^{w}$. Iterating the process, we obtain a zig-zag of $\gamma_i''$'s connecting $\alpha^*x_2$ with $\alpha^*y_2$, with $P\gamma_i''=1$, i.e.\ $\overline{\alpha^*x_2}=\overline{\alpha^*y_2}$.
In conclusion, we obtain the result:
$$
\overline{\beta_*y_1}=\overline{\beta_*x_1}=
\overline{\alpha^*x_2}=\overline{\alpha^*y_2}\,.
$$
\end{proof}

\begin{itemize}
\item Identities in $\bar\cX$. Given an object $\bar x$, its identity is the morphism
$$
\xymatrix{\bar x\ar[rr]^-{P1_x=1_{Px}}&&\bar x}\,.
$$
\item Composition in $\bar\cX$. Given a pair of composable morphisms
$$
\xymatrix{\bar x_1\ar[r]^-{\mu}&\bar x_2\ar[r]^-{\mu'}&\bar x_3}\,,
$$
their composition is the morphism
$$
\xymatrix{\bar x_1\ar[rr]^-{\mu'\cdot\mu}&&\bar x_3}\,.
$$
\end{itemize}

\begin{Lemma}
Composition in $\bar\cX$ is well defined.
\end{Lemma}
\begin{proof}
We have to prove that $\mu'\cdot\mu$ gives a legitimate morphism $\bar x_1 \to \bar x_3$ in $\bar\cX$. Clearly $\alpha_{\mu'\cdot\mu}=\alpha_{\mu'}\cdot\alpha_{\mu}=\alpha'\cdot \alpha\,$. On the other hand, $\beta_{\mu'\cdot\mu}$ is precisely the vertical comparison arrow in the following diagram:
\begin{equation}\label{diag:triangle}
\begin{aligned}
\xymatrix@C=12ex{
Px_1\ar[dr]^{\mu}\ar[d]_{\beta}
\\
m_1\ar[r]^{\hat\alpha}\ar[d]_{\alpha^*\beta'}
&Px_2\ar[dr]^{\mu'}\ar[d]_{\beta'}
\\
m_1'\ar[r]_{\hat\alpha}
&m_2\ar[r]_{\hat\alpha'}
&Px_3
}
\end{aligned}
\end{equation}
Hence, we want to prove
$$
\overline{(\alpha^*(\beta')\cdot\beta)_*x_1}=\overline{(\alpha'\cdot\alpha)^*x_3}\,.
$$
Let us describe the situation by means of the following diagram in $\cX$
\begin{equation}\label{diag:composition_in_bar_X}
\xymatrix@C=3ex{
x_1\ar[r]^-{\hat\beta}
&\beta_*x_1
&\ar[l]_{\omega_0}\cdots\ar[r]^{\omega_1}
&\alpha^*x_2\ar[r]^-{\hat\alpha}
&x_2\ar[r]^-{\hat\beta'}
&\beta'_*x_2
&\ar[l]_{\omega'_1}\cdots\ar[r]^{\omega'_0}
&\alpha'^*x_3\ar[r]^-{\hat\alpha'}
&x_3}
\end{equation}
and analyze the composite $\hat\beta'\cdot\hat\alpha$.
Indeed we can lift to $\cX$ the rectangle in diagram (\ref{diag:triangle})  in this way: first we lift $\alpha$ and $\beta'$ at $x_2$, then we lift again $\alpha$ at $\beta'_*x_2$, and  $\alpha^*(\beta')$ at $\alpha^*x_2$. The comparison we obtain will be denoted by $\omega$, and obviously $P\omega=1$. The following diagram represents the process we have just described:
\begin{equation}\label{diag:alpha*beta}
\begin{aligned}
\xymatrix@C=12ex{
\alpha^*x_2\ar[r]^{\hat\alpha}\ar[d]_{\widehat{\alpha^*\beta'}}
&x_2\ar[dd]^{\hat\beta'}
\\
(\alpha^*\beta')_*\alpha^*x_2\ar[d]_{\omega}
\\
\alpha^*\beta'_*x_2\ar[r]_{\hat\alpha}
&\beta'_*x_2
}
\end{aligned}
\end{equation}
Substituting this factorization in diagram (\ref{diag:composition_in_bar_X}), we can treat the resulting diagram in a similar way as we did in the proof of Lemma \ref{lm:bar_X_morphs_well_defined}:
\begin{equation} \label{diag:bar_X_morphs_composiiton}
\begin{aligned}
\xymatrix@C=3.5ex@R=4ex{
x_1\ar[rr]^{\hat\beta}
&&\beta_*x_1\ar[rr]^-{\widehat{\alpha^*\beta'}}&&(\alpha^*\beta')_*\beta_*x_1
\\
&&\vdots\ar[u]^-{\omega_0}&&\vdots\ar[u]_-{\gamma_0}
\\
&&z\ar[u]^-{\omega_2}\ar[d]_-{\omega_1}\ar[rr]^{\widehat{\alpha^*\beta'}}
&&\ar[u]_-{\gamma_2}\ar[d]^-{\gamma_1}
\\
&&\alpha^*x_2\ar[rr]_-{\widehat{\alpha^*\beta'}}
&&(\alpha^*\beta')_*\alpha^*x_2\ar[d]^{\omega}
\\
&&
&
&\alpha^*\beta'_*x_2\ar[rr]^{\hat\alpha}
&&\beta'_*x_2
\\
&&&&\ar[u]^-{\gamma_1'}\ar[d]_-{\gamma_2'}\ar[rr]_{\hat\alpha}
&&w\ar[u]_-{\omega_1'}\ar[d]^-{\omega_2'}
\\
&&&&\vdots\ar[d]_-{\gamma_0'}
&&\vdots\ar[d]^-{\omega_0'}
\\
&&&&\alpha^*\alpha'^*x_3 \ar[rr]_{\hat\alpha}
&&\alpha'^*x_3\ar[rr]_{\hat\alpha'}
&&x_3
}
\end{aligned}
\end{equation}
and show that $\overline{(\alpha^*\beta')_*\beta_*x_1}=\overline{\alpha^*\alpha'^*x_3}$.
As far as the right-hand side term of the equation is concerned, since the fibration $G$ is split, one can always choose the liftings of $\alpha$ and  $\alpha'$ in such a way  that the canonical comparison $\alpha^*\alpha'^*x_3\to(\alpha'\cdot\alpha)^*x_3$ is $P$-vertical, hence $\overline{\alpha^*\alpha'^*x_3}=\overline{(\alpha'\cdot\alpha)^*x_3}$. On the other hand, although it is not possible to proceed in the same way for the left-hand side of the equation, one still gets a $P$-vertical comparison
$$
\xymatrix@C=12ex{(\alpha^*\beta'\cdot\beta)_*x_1\ar[r]^{\omega'}&(\alpha^*\beta')_*\beta_*x_1}
$$
so that $\overline{(\alpha^*\beta'\cdot\beta)_*x_1}=\overline{(\alpha^*\beta')_*\beta_*x_1}$\,. By assembling together the equations
just obtained, we conclude our proof.
\end{proof}

\subsection{The functors $Q$, $\bar F$ and $\bar P$}

Let us first define the functors $Q$ and $\bar F$, and afterwards we will show that $\bar F$ is a fibration and that $Q$ is a cartesian functor, i.e., the following diagram is a morphism in $\Fib(\cA)$:
$$
\xymatrix{
\cX\ar[dr]_{F}\ar[rr]^{Q}&&\bar\cX\ar[dl]^{\bar F}\\&\cA}
$$

The functor $Q\colon \cX\to \bar\cX$ is defined as follows:
$$
Q(\xymatrix{x_1\ar[r]^{\xi}&x_2})=\xymatrix{\bar x_1\ar[r]^{P\xi}&\bar x_2}\,.
$$
This definition is consistent. Actually, in order to prove this statement, it suffices to compute the following factorization of $\xi$:
\begin{equation}\label{diag:standard_factorization}
\begin{aligned}
\xymatrix{
x_1\ar@/^3ex/[drr]^{\xi} \ar@{-->}[dr]^{\nu}\ar[d]_{\widehat{P\nu}}
\\
\ar@{-->}[r]_{\omega}
&\ar[r]_{\widehat{F\xi}}
&x_2
}
\end{aligned}
\end{equation}
first factor $\xi$ through a cartesian lifting at $x_2$ of $\alpha_{P\xi}=F\xi$ in order to get an $F$-vertical arrow $\nu$, then factor $\nu$ through an opcartesian lifting at $x_1$ of $\beta_{P\xi}=P\nu$ in order to get a $P$-vertical arrow $\omega$, as in the diagram above.  Therefore it is clear that:
$$
\overline{(\beta_{P\xi})_*x_1}=\overline{(P\nu)_*x_1}=\overline{(F\xi)^*x_2}=\overline{(\alpha_{P\xi})^*x_2}
$$

The functor $\bar F\colon \bar\cX\to\cA$ is defined as follows:
$$
\bar F(\xymatrix{\bar x_1\ar[r]^{\mu}&\bar x_2})=\xymatrix{Fx_1\ar[r]^{G\mu}&Fx_2}\,.
$$
The following statement was expected.
\begin{Proposition}\label{prop:bar_F_fibration}
$(\bar \cX,\bar F)$ is an object of $\Fib(\cA)$.
\end{Proposition}
\begin{proof}
Given an object $\bar x_2$ of $\bar\cX$, and an arrow $\xymatrix{ a_1\ar[r]^-{\alpha} &\bar F\bar x_2=Fx_2}$, first lift $\alpha$ at $x_2$, and then obtain from this lifting $\hat\alpha\colon x_1\to x_2$ the morphism of $\bar\cX$
$$
Q(\hat\alpha) = \xymatrix@C=12ex{\bar x_1\ar[r]^{P\hat\alpha}&\bar x_2}\,.
$$
Let us prove that such a morphism is cartesian.

Suppose we are given a morphism $\mu'\colon\bar x_0\to\bar x_2$ of $\bar \cX$ such that  $\bar F\mu'=\alpha\cdot \alpha''$ in $\cA$:
$$
\xymatrix@C=13ex{
\bar x_0\ar@/^2ex/[drr]^{\mu'}
\\
&\bar x_1\ar[r]_{P\hat\alpha}
&\bar x_2&\bar\cX
\\
Px_0\ar@/^2ex/[drr]^{\mu'}\ar@{-->}[dr]_{\mu''}
\\
&Px_1\ar[r]_{P\hat\alpha}&Px_2&\cM
\\
Fx_0\ar[r]_{\alpha''}
&Fx_1\ar[r]_-{\alpha}
&Fx_2&\cA
}
$$
Since $P\hat\alpha$ is $G$-cartesian, there exists a unique $\mu''$ such that $P\hat\alpha\cdot\mu''=\mu'$ and $G\mu''=\alpha''$. This is the arrow $\bar x_0\to\bar x_1$ we are looking for, provided we show that it is a legitimate arrow of $\bar\cX$. This follows from the calculation below, right after the explicative diagram:
$$
\xymatrix{
Px_0\ar@/^3ex/[drr]^{\mu'}\ar[dr]^{\mu''}\ar[d]_{\beta_{\mu''}=\beta_{\mu'}}
\\
\ar[r]_-{P\hat\alpha''}
&Px_1\ar[r]_-{P\hat\alpha}&Px_2&\cM
\\
Fx_0\ar[r]_{\alpha''}
&Fx_1\ar[r]_{\alpha}
&Fx_2&\cA
}
$$
$$
\overline{(\alpha_{\mu''})^*x_1}
=\overline{\alpha''^*x_1}
=\overline{(\alpha \cdot\alpha'')^*x_2}
=\overline{(\alpha_{\mu'})^*x_2}
=\overline{(\beta_{\mu'})_*x_0}
=\overline{(\beta_{\mu''})_*x_0}
$$
where we used the fact that the fibration $G$ is split.
\end{proof}

\begin{Proposition}
$Q\colon (\cX,F)\to (\bar \cX,\bar F)$ is an arrow of $\Fib(\cA)$.
\end{Proposition}
\begin{proof}
It is an obvious corollary to the following characterization.
\end{proof}

\begin{Proposition}\label{prop:characterization_cartesian_in_bar_X}
The morphism  $\xymatrix{\bar x_1\ar[r]^{\mu}&\bar x_2}$
is $\bar F$-cartesian if and only if the arrow $\xymatrix{Px_1\ar[r]^{\mu}&Px_2}$
is $G$-cartesian.
\end{Proposition}

\begin{proof}
The sufficient condition follows easily from the proof of Proposition \ref{prop:bar_F_fibration}. As far as the necessary condition is concerned, let us be given an $\bar F$-cartesian arrow $\xymatrix{\bar x_1\ar[r]^{\mu}&\bar x_2}$, and consider the cartesian lifting of $G\mu$ at $x_2$. Of course, $P$ sends such a lifting to a $G$-cartesian arrow $P(\widehat{G\mu})$. Therefore, there exists a unique $\tau\colon Px_1\to Px_1'$ in $\cM$ such that $P(\widehat{G\mu})\cdot\tau=\mu$ and $G\tau=1$. It is easy to see that such a $\tau$ defines a morphism $\bar x_1\to\bar x_1'$ in $\bar\cX$. On the other hand, since $\mu$ is $\bar F$-cartesian, there exists a unique $\sigma\colon\bar x_1'\to \bar x_1$ such that $\mu\cdot\sigma=P(\widehat{G\mu})$ in $\bar\cX$ and $\bar F\sigma=1$.
Therefore, since $P(\widehat{G\mu})$ is $G$-cartesian, we deduce $\tau\cdot\sigma=1_{Px'_1}$, and since $\mu$ is $\bar F$-cartesian, $\sigma\cdot\tau=1_{\bar x_1}$, which in turn implies $\sigma\cdot\tau=1_{Px_1}$. In conclusion, $\tau$ is a $G$-vertical isomorphism, and hence $\mu\colon Px_1\to Px_2$ is $G$-cartesian.
\end{proof}


Let us define the functor $\bar P\colon \bar\cX\to \cM$ as follows:
$$
\bar P(\xymatrix{\bar x_1\ar[r]^{\mu}&\bar x_2})=\xymatrix{Px_1\ar[r]^{\mu}&Px_2}\,.
$$
By Proposition \ref{prop:characterization_cartesian_in_bar_X}, it is clearly cartesian, and an easy computation shows $\bar P	\cdot Q=P$. Moreover, the following statement holds.

\begin{Proposition}\label{prop:bar_P_opfib}
$\bar P\colon (\bar\cX,\bar F)\to(\cM,G)$ is a discrete opfibration in $\Fib(\cA)$.
\end{Proposition}

\begin{proof}
Recall that, thanks to Corollary \ref{cor:disc_opfib_in_Fib_A}, this amounts to prove that, for every object $a$ of $\cA$, the restriction of $\bar P$ to the fiber over $a$
$$
\bar P_a\colon \xymatrix{\bar\cX_a\ar[r]&\cM_a}
$$
is a discrete opfibration. In fact, from our general hypotheses on $P$, we know that the restriction of $P$ to the fiber over $a$
$$
P_a\colon \xymatrix{\cX_a\ar[r]&\cM_a}
$$
is an opfibration, therefore, given an object $\bar x_1$ of $\bar \cX_a$, and a morphism $\mu\colon \bar P\bar x_1 \to m_2$ in $\cM_a$, there exists an opcartesian lifting $\hat\mu\colon x_1\to x_2$ of $\mu$ at $x_1$. Hence
$$
Q(\hat\mu)=\xymatrix{\bar x_1\ar[r]^{\mu}&\bar x_2}
$$
lifts $\mu$ at $\bar x_1$. In fact, such a lifting is unique since $\bar P$ is faithful by definition.
\end{proof}

\subsection{The comprehensive factorization of $P$ in $\Fib(\cA)$}

By Proposition \ref{prop:bar_P_opfib}, $\bar P$ is a discrete opfibration in the 2-category $\Fib(\cA)$. Therefore, in order to prove that $(Q,\bar P)$ is the comprehensive factorization of $P$, we have to prove the following theorem.

\begin{Theorem} \label{thm:Q_initial}
The functor $Q$ is initial in $\Fib(\cA)$.
\end{Theorem}

In fact, we are going to prove that $Q$ is the coidentifier of the identee of $P$ and we will get the result thanks to Proposition \ref{prop:coid_final_initial} (the reader unfamiliar with such 2-categorical notions will find it useful to consult the Appendix). Our strategy will be to prove the theorem by means of a sequence of lemmas. Here is the first one.

\begin{Lemma}\label{lm:lemma_1}
The identee of $P$ in $\CAT$ is also the identee of $P$ in $\CAT/\cA$ and in $\Fib(\cA)$.
\end{Lemma}
\begin{proof}
Let $\kappa\colon D\Rightarrow C\colon \cI(P)\to \cX$ be the identee of $P$ in $\CAT$. By the description of identees in terms of a pullback (see diagram  (\ref{diag:identee})), it is clear that identees are stable under slicing, so that $\kappa$ is also the identee of $P$ in $\CAT/\cA$:
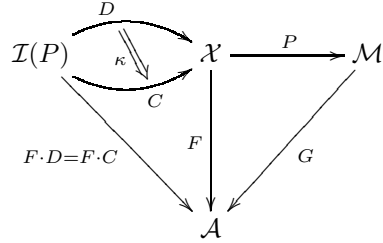
$$
\xymatrix@C=10ex@R=12ex{
\cI(P)\ar@/^3ex/[r]^(.4)D_(.45){}="1"\ar@/_3ex/[r]_(.65)C^(.65){}="2" \ar[dr]_{F\cdot D=F\cdot C}\ar@{=>}"1";"2"_\kappa&\cX\ar[r]^{P}\ar[d]_{F}&\cM\ar[dl]^{G}\\&\cA}
$$
In order to prove that $\kappa$ is in fact the identee of $P$ also in $\Fib(\cA)$, it suffices to prove that
\begin{itemize}
\item[(i)] $FD=FC$ is a fibration,
\item[(ii)] $C$ and $D$ are cartesian functors,
\item[(iii)] $\kappa$ is a $F$-vertical natural transformation.
\end{itemize}
Hence let us describe the left hand side of the diagram above. We start with the description of the category $\cI(P)$.

The objects of $\cI(P)$ are the arrows of $\cX$
$$
\xymatrix{x\ar[r]^{\omega}& y}
$$
such that $P\omega=1$. The morphisms $(\xi,\upsilon)\colon \omega_1\to\omega_2$ of $\cI(P)$ are pairs $(\xi,\upsilon)$ of arrows  of $\cX$ such that $\upsilon\cdot\omega_1=\omega_2\cdot\xi$. In other words, $\cI(P)$ is the full subcategory of the comma category $(\cX\downarrow\cX)$ determined by the $P$-vertical arrows.

$D$, $C$ and $\kappa$ are obviously defined by:
$$
D\colon(\xi,\upsilon)\mapsto \xi\,,\qquad C\colon(\xi,\upsilon)\mapsto \upsilon\,,\qquad \kappa_{\omega}=\omega\,.
$$
Since $F\omega_i=GP\omega_i=G1=1$, for $i=1,2$, then clearly
$FD(\xi,\upsilon)=F\xi=F\upsilon=FC(\xi,\upsilon)$. Indeed, the functor
$FD=FC$  is a fibration: the reader can check this claim with the help of the following diagram:
$$
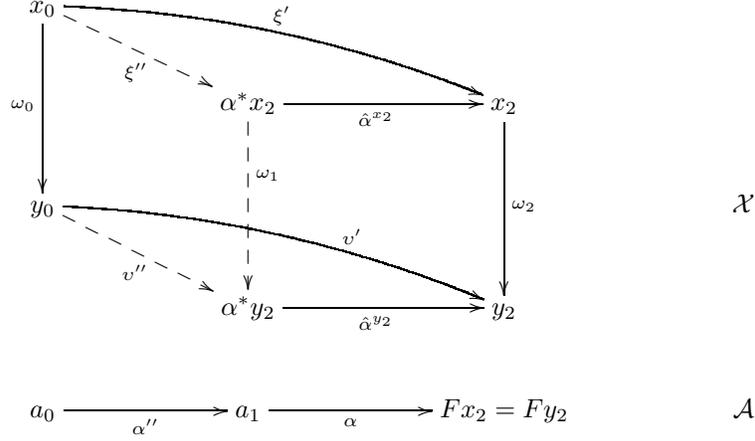
\xymatrix@C=13ex{
x_0\ar@/^2ex/[drr]^{\xi'}\ar@{-->}[dr]_{\xi''}\ar[dd]_{\omega_0}
\\
&\alpha^*x_2\ar[r]_{\hat\alpha^{x_2}}\ar@{-->}[dd]^(.35){\omega_1}
&x_2\ar[dd]^{\omega_2}
\\
y_0\ar@/^2ex/[drr]^(.65){\upsilon'}\ar@{-->}[dr]_{\upsilon''}&&&\cX
\\
&\alpha^*y_2\ar[r]_{\hat\alpha^{y_2}}&y_2
\\
a_0\ar[r]_{\alpha''}
&a_1\ar[r]_-{\alpha}
&Fx_2=Fy_2&\cA
}
$$
Hence (i) is proved.

As far as (ii) is concerned, it follows from the fact that $FD$-cartesian morphisms in $\cI(P)$ are precisely those morphisms $(\xi,\upsilon)$ with $F$-cartesian components. This can be easily proved by direct computation, by comparing the cartesian morphism $(\xi,\upsilon)$ with the morphism $(\hat\alpha^{x_2},\hat\alpha^{y_2})$ obtained by lifting $F\xi=F\upsilon$.

Finally, (iii) is obvious, since $F\kappa_\omega=F\omega=1$.
\end{proof}

Now we are ready to prove the following.
\begin{Lemma}\label{lm:coid_in_Cat}
The functor $Q$ is the coidentifier of $\kappa$ in $\CAT$.
\end{Lemma}
\begin{proof}
By definition of $Q$, given an object $\omega\colon x\to y$ of $\cI(P)$, one computes
$$
(Q\kappa)_{\omega}=Q(\kappa_\omega)=Q(\omega)=1_{\bar x}=1_{(QD)\omega}\,.
$$
In other words,  $Q$ coidentifies $\kappa$.

Let us check the universal property. To this end, let us consider the following diagram of solid arrows in $\CAT$:
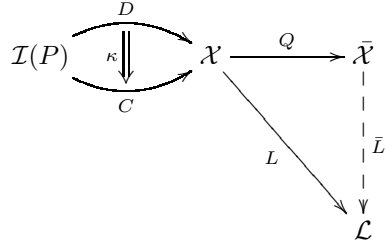
$$
\xymatrix@C=10ex@R=12ex{
\cI(P)\ar@/^3ex/[r]^D_{}="1"\ar@/_3ex/[r]_C^{}="2" \ar@{=>}"1";"2"_\kappa&\cX\ar[r]^{Q}\ar[dr]_{L}&\bar\cX\ar@{-->}[d]^{\bar L}\\&&\cL}
$$
with $L\kappa=id_{LD}$. Our aim is to define a functor $\bar L\colon \bar \cX\to \cL$ such that $\bar LQ=L$, and prove that it is unique with this property.

On objects, the obvious definition is $\bar L(\bar x)=Lx$. In fact, it is well posed, since if $y\in \bar x$, then there is a zig-zag of arrows $\omega_i$ connecting $y$ with $x$, which are turned into identities by $L$, so that $Ly=Lx$.

On morphisms the situation is more sensible. We let
$$
\bar L(\xymatrix{\bar x_1\ar[r]^-\mu&\bar x_2})= L(\hat\alpha)\cdot L(\hat\beta)\,,
$$
where $\alpha=\alpha_{\mu}$ and $\beta=\beta_{\mu}$ (see Section \ref{sec:bar_X} for details). In fact, by Lemma \ref{lm:bar_X_morphs_well_defined}, also this definition is well posed. Moreover, for every object $\bar x$ of $\bar\cX$,
$$
\bar L(1_{\bar x})= 1_{Lx}=1_{\bar L\bar x}\,.
$$
Therefore, what remains to be proved is that $\bar L$ preserves the composition of morphisms of $\bar\cX$.

To this end, let us consider a pair of composable morphisms of $\bar\cX$:
$$
\xymatrix{
\bar x_1\ar[r]^{\mu}&\bar x_2\ar[r]^{\mu'}&\bar x_3
}
$$
together with their decomposition as in diagram (\ref{diag:composition_in_bar_X}). We are to prove that:
$$
L\hat\alpha'\cdot L\hat\beta'\cdot L\hat\alpha\cdot L\hat\beta=L(\widehat{\alpha'\cdot\alpha})\cdot L(\widehat{\alpha^*\beta'\cdot\beta})\,.
$$
where $\alpha^*\beta'$ is defined in diagram (\ref{diag:triangle}). Indeed, applying $L$ to diagram (\ref{diag:alpha*beta}) yields
\begin{equation}\label{eq:intermediate_step}
L\hat\beta'\cdot L\hat\alpha=L\hat\alpha\cdot L(\widehat{\alpha^*\beta'}).
\end{equation}
The rest of the proof follows from a careful analysis of diagram (\ref{diag:bar_X_morphs_composiiton}).

As far as the right-hand side of this diagram is concerned, one observes (for the top commutative rectangle) that, since $L(\gamma_1'')=1$ and $L(\omega_1'')=1$,
$$
L\hat\alpha^{\beta'_*x_2} =L\hat\alpha^{w}\,.
$$
This process can be iterated in order to show that, applied to any of the cartesian lifting of $\alpha$ represented in diagram (\ref{diag:bar_X_morphs_composiiton}), $L$ always assumes the same value $L\hat\alpha$.
For the same reason, as far as the left hand side of the diagram is concerned, applied to any of the cartesian lifting of $\alpha^*\beta'$, $L$ always assumes the same value $L(\widehat{\alpha^*\beta'})$.
Therefore we have the calculation:
$$
L\hat\alpha'\cdot L\hat\beta'\cdot L\hat\alpha\cdot L\hat\beta=
L\hat\alpha'\cdot L\hat\alpha\cdot L(\widehat{\alpha^*\beta'})\cdot L\hat\beta=
$$
$$
=L(\widehat{\alpha'\cdot\alpha})\cdot L(\widehat{\alpha^*\beta'})\cdot L\hat\beta
=L(\widehat{\alpha'\cdot\alpha})\cdot L(\widehat{\alpha^*\beta'\cdot\beta})\,.
$$

The uniqueness is proved in this way. Let us suppose there is another functor $L'\colon \bar\cX\to \cL$ such that $L'\cdot Q=L$. Since $\bar F$ is a fibration, any morphism $\xymatrix{\bar x_1\ar[r]^{\mu}&\bar x_2}$ of $\bar X$ can be decomposed into a pair $(\beta,\hat\alpha)$,  where $\hat\alpha$ is a $\bar F$-cartesian lifting of $\alpha=\bar F\mu$ at $\bar x_2$ and $\beta$ is the $\bar F$-vertical comparison, then it suffices to check $L'=\bar L$ on such classes of arrows. Actually, as seen in Proposition \ref{prop:bar_F_fibration}, $\hat\alpha=Q\hat\alpha^{x_2}$, so that
$$
\bar L\hat \alpha=
\bar LQ\hat\alpha^{x_2}=
L'Q\hat\alpha^{x_2}=
L'\hat \alpha\,.$$
On the other hand, given an $\bar F$-vertical arrow $\beta$, $1=\bar F\beta=G(\bar P\beta)$, i.e. $G\beta=1$. Therefore we can consider an opcartesian lifting $\hat\beta^{x_1}$ of such an arrow at $x_1$, and conclude the proof:
$$
\bar L\beta=\bar LQ\hat\beta^{x_1}
= L'Q\hat\beta^{x_1}=L'\beta \,.
$$
\end{proof}
\begin{Lemma}\label{lm:coid_in_CatA}
The functor $Q$ is the coidentifier of $\kappa$ in $\CAT/\cA$.
\end{Lemma}
\begin{proof}
Let us consider the following diagram of solid arrows in $\CAT/\cA$:
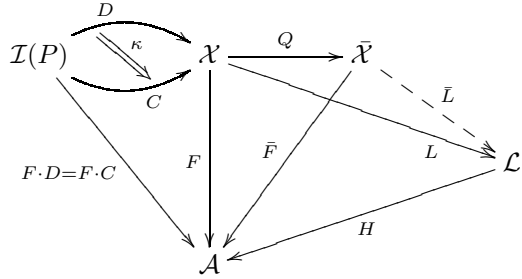
$$
\xymatrix@C=10ex@R=6ex{
\cI(P)\ar@/^3ex/[r]^(.4)D_(.30){}="1"\ar@/_3ex/[r]_(.65)C^(.70){}="2" \ar[ddr]_{F\cdot D=F\cdot C}\ar@{=>}"1";"2"^\kappa
&\cX\ar[r]^{Q}\ar[dd]_{F}\ar[drr]_(.75){L}
&\bar\cX\ar[ddl]_{\bar F}\ar@{-->}[dr]^{\bar L}
\\
&&&\cL\ar[dll]^{H}
\\
&\cA}
$$
with
$$
L\cdot \kappa=id_{L\cdot D}\colon L\cdot D\Rightarrow L\cdot C\colon (\cI(P),F\cdot D)\to (\cL,H)\,.
$$
This implies that, for every object $\omega\colon x\to y$ of $\cI(P)$, $L(\kappa_\omega)=L\omega=1_{Lx}$, i.e.\ $L\cdot\kappa=id$ in $\CAT$. By Lemma \ref{lm:coid_in_Cat} above, there exists a unique $\bar L$ such that $\bar L\cdot Q=L$. But clearly $\bar L$ is a morphism in $\CAT/\cA$, since $H\cdot \bar L\cdot Q=F=\bar F\cdot Q$ and $Q$ is couniversal.
\end{proof}
\begin{Lemma}\label{lm:coid_in_FibA}
The functor $Q$ is the coidentifier of $\kappa$ in $\Fib(\cA)$.
\end{Lemma}
\begin{proof}
We can repeat the argument used in the proof of Lemma \ref{lm:coid_in_CatA}, but starting with fibrations $F$, $\bar F$ and $H$, and cartesian functors $Q$ and $L$. What remains to prove is that the functor $\bar L$ as above is cartesian. Now, by Proposition \ref{prop:characterization_cartesian_in_bar_X}, $\bar F$-cartesian morphisms in $\bar\cX$ can be lifted to $F$-cartesian morphisms in $\cX$. Therefore their image under $\bar L$ coincides with the image under $L$ of an $F$-cartesian morphism of $\cX$. But $L$ is cartesian, and the proof is done.
\end{proof}

\section{Some selected examples} \label{sec:Examples}

Relevant examples of the structures investigated in the present paper arise naturally in the study of (not necessarily abelian) cohomological algebra.

As a matter of fact, a classical interpretation of cohomology coefficients in group cohomology are the so-called \emph{crossed $n$-fold extensions of groups}. They are a natural generalization of Yoneda's $\Ext^n$, therefore their assessment in our setting is in order.

\subsection{Crossed extensions of groups} \label{sec:groups}

Crossed modules of groups have been introduced by J.H.C. Whitehead in \cite{Whi}. Later, they have been used by Mac Lane and Whitehead to represent the third cohomology group  of a multiplicative group $C$ with coefficients in a $C$-module $B$  (see \cite{MLW}). To this end, it is more appropriate to consider crossed modules with chosen kernel and cokernel, i.e.\ what is commonly known as \emph{crossed extension}.

\begin{Remark}{\em
The informed reader could argue that we should have started with the simpler example of group extensions with abelian kernel. However, as it will be clearer later in this section, there are reasons for which that example is \emph{too simple} to start with, and its simplicity hides features that can be appreciated only for crossed extensions of dimension $n\geq 2$. This is true also in  the abelian case, where the similarity relation for $n=1$ is already an equivalence relation:  by the \emph{five lemma},  all morphisms of extensions with fixed kernel and cokernel are isomorphisms.
 }\end{Remark}

A crossed extension of groups is an exact sequence in $\Gp$
\begin{equation}
X\colon\qquad
\xymatrix{
0\ar[r]
& B\ar[r]^j
& G_2\ar[r]^{\partial}
& G_1\ar[r]^p
& C\ar[r]
& 1}
\end{equation}
where the central homomorphism is a crossed module. Recall that this means that $G_1$ acts on $G_2$ in such a way that, for every $g_1\in G_1$ and $g_2,g_2' \in G_2$, one has:
$$
(i)\ \ \partial(g_1*g_2)= g_1 \partial(g_2) g_1^{-1}\,,\qquad (ii) \ \ \partial(g_2)*g_2'=g_2+g_2'-g_2\,,
$$
where the group operation of $G_2$ is written additively, although $G_2$ is not necessarily abelian.
Recall also that crossed modules are always proper homomorphisms, i.e.\ $\partial$ factors as a regular epimorphism (a surjection, in this case) followed by a normal monomorphism, so that every crossed module together with its kernel and its cokernel always yields an exact sequence. It is easy to see that $B$ is contained in the centre of $G_2$, and therefore the action of $G_1$ on $G_2$ induces a natural $C$-module structure on the abelian group $B$.

A morphism of crossed extensions $(\gamma,f_1,f_2,\beta)\colon X\to X' $ is a morphism of exact sequences
$$
\xymatrix{
X\colon\ar@<-1ex>[d]&
0\ar[r]
& B\ar[r]^j\ar[d]_\beta
& G_2\ar[r]^{\partial}\ar[d]_{f_2}
& G_1\ar[r]^p\ar[d]^{f_1}
& C\ar[r]\ar[d]^{\gamma}
& 1
\\
X'\colon&
0\ar[r]
& B'\ar[r]_{j'}
& G'_2\ar[r]_{\partial'}
& G'_1\ar[r]_{p'}
& C'\ar[r]
& 1
}
$$
such that $(f_1,f_2)$ is a morphism of crossed modules, i.e., equivariant with respect to the actions. As a consequence, the pair $(\gamma,\beta)$ is a morphism of group-modules.

Let us denote by $\XExt$ the category of crossed extensions of groups and their morphisms, and by $\Mod$ the category of group-modules and their morphisms. In fact, we have just defined the forgetful  functor $\Pi\colon \XExt\to\Mod$, where $\Pi(\gamma,f_1,f_2,\beta)=(\gamma,\beta)$.

\begin{Theorem}\label{thm:fib_opfib_XModGp}
The commutative diagram of categories and functors
\begin{equation} \label{diag:xext_triang}
\begin{aligned}
\xymatrix{\XExt\ar[rr]^\Pi\ar[dr]_{\Pi_0}&&\Mod\ar[dl]^{(\underline{\ })_0}\\&\Gp}
\end{aligned}
\end{equation}
is a fiberwise opfibration in $\Fib(\Gp)$,
where $\Pi$ has been defined above, and $(\underline{\ })_0$ is the functor that assigns the group $C$ to the $C$-module $B$.
\end{Theorem}
\begin{proof}
We only sketch the proof, since all the ingredients are well-known facts in the literature concerning crossed modules and extensions.

\begin{itemize}
\item[(i)] $\Pi_0$ is a fibration. Indeed, given a crossed extension $X'$ as above, and a group homomorphism $\gamma \colon C\to C'$, its cartesian lifting  $\hat\gamma\colon \gamma^*X'\to X'$  is given by the following construction:
$$
\xymatrix{
\gamma^*X'\colon\ar@<-1ex>[d]&
0\ar[r]
& B'\ar[r]^{j'}\ar@{=}[d]
& G'_2\ar[r]^-{\langle\partial',0\rangle}\ar@{=}[d]
& G_1'\times_{C'}C\ar[r]^-{pr_2}\ar[d]_{pr_1}
& C\ar[r]\ar[d]^{\gamma}
& 1
\\
X'\colon&
0\ar[r]
& B'\ar[r]_{j'}
& G'_2\ar[r]_{\partial'}
& G_1'\ar[r]_{p'}
& C'\ar[r]
& 1
}
$$
where the rightmost square is a pullback in $\Gp$ (sometimes called fibered product of groups). Calculations show that the comparison $\langle\partial',0\rangle$ inherits a crossed module structure such that the four-tuple  $(\gamma,pr_1,1,1)$ is a morphism of crossed extensions, cartesian with respect to
$\Pi_0$.
Let us notice that the fibration $\Pi_0$ is not split.
\item[(ii)] $(\underline{\ })_0$ is a split fibration. This is straightforward, since  for a $C'$-module $B'$, and a group homomorphism $\gamma\colon C\to C'$, one obtains a cartesian lifting by precomposition:
$$
\xymatrix{\gamma^*M'\colon\quad \ar@<-1.5ex>[d]_{(\gamma,1)}& \ar[d]_{\gamma\times 1}C\times B'\ar[r]& B'\ar@{=}[d] \\
M'\colon\quad &C'\times B'\ar[r]& B'}
$$
\item[(iii)] One easily checks that $\Pi$ is a cartesian functor.
\item[(iv)] Let us fix a group $C$; the restriction to the fiber $\Pi_C\colon \XExt_C\to\Mod_C$ is an opfibration.
In order to prove this, let us consider a crossed extension $X$ as above and a $C$-module homomorphism $(1_C,\beta)\colon B\to B'$. In this case, the strategy to produce an opcartesian lifting of $\beta$ at $X$ is similar to the one we realized  in (i), but not strictly dual as one might think. As a matter of fact, if we would compute a pushout of $j$ along $\beta$ (sometimes called  amalgamate sum of groups), we would possibly lose left exactness of the resulting sequence, because the pushout of a monomorphism is not a monomorphism in general. This fact does not happen in the abelian setting, and this is why Yoneda could set up the opfibration dualizing the fibration. The categorical construction we need here is called \emph{push forward} in \cite{pf}; it generalizes the one proposed by Mac Lane in \cite[Exercise IV.4.3]{Homology}, and in this case it specializes the one described in \cite{Behrang} for $j$ any crossed module. Explicitly, let us consider the pair $(\rho,B'\times^BG_2)$, defined by the following short exact sequence of groups:
$$
\xymatrix@C=6ex{
0\ar[r]
&B\ar[r]^-{\langle \beta,-j\rangle}
&B'\times G_2\ar[r]^-\rho
&B'\times^BG_2\ar[r]
&0}
$$
where the normal monomorphism $\langle \beta,-j\rangle$ is given by the assignment
$$
b\mapsto (\beta(b),-j(b))\,.
$$
The following diagram displays the relevant constructions.
$$
\xymatrix@C=3ex@R=3ex{
X\ar@<-1ex>[dd]\colon&
0\ar[r]
&B\ar[rr]^j\ar[dd]_{\beta}
&&G_2\ar[r]^{\partial}\ar[dl]_{\iota_2}\ar[dd]^{\rho\cdot\iota_2}
&G_1\ar[r]^{p}\ar@{=}[dd]
&C\ar[r]\ar@{=}[dd]
&1
\\
&&&B'\!\times\!G_2\ar[dr]^{\rho}
\\
\beta_*X\colon
&0\ar[r]
&B'\ar[rr]_{\rho\cdot\iota_1} \ar[ur]^{\iota_1}
&&B'\times^BG_2\ar[r]_-{\delta}
&G_1\ar[r]_p
&C\ar[r]
&1
}
$$
On the left-hand side of the diagram, the arrows $\iota_1$ and $\iota_2$ are given by the assignments $\iota_1(b')=(b',0)$ and  $\iota_2(g_2)=(0,g_2)$. Of course there is no reason why $\iota_1\cdot \beta=\iota_2\cdot j$, in general. However, the square of vertexes $B$, $G_2$, $B'$ and $B'\times^BG_2$ does commute, since in fact $\rho$ is the coequalizer of the pair $(\iota_1\cdot\beta,\iota_2\cdot j)$. The homomorphism $\rho\cdot\iota_1$  is the \emph{push forward} of $j$ along $\beta$ \cite{pf}.  It is possible to prove that the $\rho\cdot\iota_1$  is in fact a normal monomorphism and that it has the same (isomorphic) cokernel as $j$ so that we can obtain a homomorphism $\delta$ such that $\delta\cdot\rho\cdot\iota_2=\partial$. Moreover, $\delta$ is a crossed module, with action induced from both the $C$-module structure of $B'$ and the crossed module structure of $\partial$. Namely, for $g_1\in G_1$ and $[b',g_2]\in B'\times^B G_2$, we have (classes denoted by square brackets):
$$
g_1*[b',g_2]=[p(g_1)*b', g_1*g_2] \,.
$$
Routine calculations show that the crossed extension $\beta_*X=(\rho\cdot\iota_1,\delta,p)$ is $\Pi_C$-opcartesian.
\end{itemize}
\end{proof}
Even if the above theorem is what we need in order to have crossed extensions of groups fitting into our general scheme, it may be interesting to observe that actually $\Pi$ fulfills condition (C) of Section \ref{sec:opfib_FibA}, i.e.\ the following proposition holds.

\begin{Proposition} \label{prop:CforXextGp}
The morphism $\Pi$ of diagram (\ref{diag:xext_triang}) is an opfibration in $\Fib(\Gp)$.
\end{Proposition}

\begin{proof}
Thanks to Theorem \ref{thm:opfib_in_Fib_A} and Theorem \ref{thm:fib_opfib_XModGp} above, we only have to show that condition (C) holds for $\Pi$. Let us consider the following diagram, displaying the cartesian and opcartesian lifting of a crossed extension $X$ along a group homomorphism $\gamma$ and a $C$-module homomorphism $\beta$ respectively:
\[
\xymatrix@!=1.5ex{
& \gamma^*X \ar[dl]_{\hat\gamma} & 0 \ar[r] & B \ar@{=}[dl] \ar[dd]^(.3){\beta} \ar[rr]^-{j} & & G_2 \ar@{=}[dl] \ar[dd] \ar[rrr] & & & G_1 \times_C C' \ar[dl] \ar@{=}[dd] \ar[rr]^-{p'} & & C' \ar[dl]_(.6){\gamma} \ar@{=}[dd] \ar[r] & 1 \\
X \ar[dd]_{\hat\beta} & 0 \ar[r] & B \ar[dd]_(.3){\beta} \ar[rr]_(.3){j} & & G_2 \ar[dd] \ar[rrr]_{\partial} & & & G_1 \ar@{=}[dd] \ar[rr]_(.7){p} & & C \ar@{=}[dd] \ar[rr] & & 1 \\
& 0 \ar[rr] & & B' \ar@{=}[dl] \ar[rr]_(.3){j'} & & B' \times^{B} G_2 \ar@{=}[dl] \ar@{-->}[rrr]_(.4){\partial'} & & & G_1 \times_C C' \ar[dl] \ar[rr]^(.7){p'} & & C' \ar[dl]^(.4){\gamma} \ar[r] & 1 \\
\beta_*X & 0 \ar[r] & B' \ar[rr]_-{j'} & & B' \times^{B} G_2 \ar[rrr] & & & G_1 \ar[rr]_-p & & C \ar[r] & 1
}
\]
We have to prove that $\gamma^*\beta_*X\cong(\gamma^*\beta)_*\gamma^*X$. It suffices to observe that if we define the dashed arrow $\partial'$ as the composite of the cokernel of $j'$ with the kernel of $p'$, the sequence $(p',\partial',j')$ gives a crossed extension which is at the same time the pullback of $\beta_*X$ along $\gamma$ and the push forward of $\gamma^*X$ along $\gamma^*\beta$. This can be checked by means of the description of cartesian and opcartesian liftings provided in the proof of Theorem \ref{thm:fib_opfib_XModGp}.
\end{proof}

Once we proved that crossed extensions of groups fit into our general scheme, several issues can be considered.

To start with, let us examine for $\cX=\XExt$,  the \emph{standard factorization} of diagram (\ref{diag:standard_factorization}) applied to a morphism $(\gamma,f_1,f_2,\beta)$ of crossed extensions as above. The result is the three-fold factorization represented below:
$$
\xymatrix{
X\ar[d]_{(1,1,f_2',\beta)}
&0\ar[r]
& B\ar[r]^j\ar[d]_\beta
& G_2\ar[r]^{\partial}\ar[d]^{f'_2}
& G_1\ar[r]^p\ar@{=}[d]
& C\ar[r]\ar@{=}[d]
& 1
\\
\beta_*X\ar[d]_{(1,\omega_1,\omega_2,1)}
&0\ar[r]
& B'\ar[r]^k\ar@{=}[d]
& P_2\ar[r]^{\delta}\ar[d]_{\omega_2}
& G_1\ar[r]^p\ar[d]^{\omega_1}
& C\ar[r]\ar@{=}[d]
& 1
\\
\gamma^*X'\ar[d]_{(\gamma,f_1',1,1)}
&0\ar[r]
& B'\ar[r]_{j'}\ar@{=}[d]
& G'_2\ar[r]_{\delta'}\ar@{=}[d]
& P_1\ar[r]_{q'}\ar[d]_{f'_1}
& C\ar[r]\ar[d]^{\gamma}
& 1
\\
X'\colon&
0\ar[r]
& B'\ar[r]_{j'}
& G'_2\ar[r]_{\partial'}
& G'_1\ar[r]_{p'}
& C'\ar[r]
& 1
}
$$
where:
$$
P_1=G_1'\times_{C'}C\,,\qquad P_2=B'\times^BG_2\,.
$$
The morphism has been factored into an opcartesian one, followed by a weak equivalence, followed by a cartesian one:
$$
(\gamma,f_1,f_2,\beta)= (\gamma,f_1',1,1)\cdot (1,\omega_1,\omega_2,1)\cdot(1,1,f_2',\beta)
$$
(recall that a morphism of crossed extensions $(\gamma,f_1,f_2,\beta)$ is a weak equivalence if $\beta=1$ and $\gamma=1$).

A classification theorem for (possibly weak) morphisms of crossed extensions, based on this three-fold factorization, is the main theme of the forthcoming article \cite{Schreier}.

Notice that condition (C) means precisely that in the above factorization applied to the morphism of crossed extensions
\[
\xymatrix{
\gamma^*X \ar[r]^{\hat\gamma} & X \ar[r]^{\hat\beta} & \beta_*X
}
\]
of Proposition \ref{prop:CforXextGp}, the corresponding middle arrow $(1,\omega_1,\omega_2,1)$ is in fact an isomorphism.

Another interesting investigation that can be based on our theory is a structural analysis of the category $\bar\cX=\overline \XExt$. Following the general case already developed in Section \ref{sec:bar_X},  objects are  connected components of  $\Pi$-fibers, i.e.\ equivalence classes of crossed extensions under a   \emph{similarity relation} defined in the same way as the one introduced by Yoneda.
Given two crossed extensions $X$ and $X'$ as above, an arrow $\bar X\to\bar X'$ can be described as a morphism of group-modules $\Pi X\to \Pi X'$, satisfying the condition that we are going to spell out. Let us be given a morphism of group-modules $(\gamma,\beta)$. Using the standard factorization determined by the fibration $(\underline{\ } )_0$, we immediately obtain  $(\gamma,\beta)= (\gamma,1)\cdot(1,\beta)$:
$$
\xymatrix{
M\ar[d]_{(1,\beta)}&\ar[d]_{1\times\beta}C\times B\ar[r]& B\ar[d]^{\beta}
\\
\gamma^*M'\ar[d]_{(\gamma,1)}&\ar[d]_{\gamma\times 1}C\times B'\ar[r]& B'\ar@{=}[d]
\\
M'&C'\times B'\ar[r]& B'
}
$$
Now, recall that any equivalence class of crossed extensions $\bar X$ can be identified with an element $\chi$ of the third cohomology group $H^3(C,B)$, where $B$ is the $C$-module determined by $X$. Then the group module morphism $(\gamma,\beta)$ gives a morphism $\bar X\to\bar X'$ in the sense of Section \ref{sec:bar_X} if and only if the following cohomological condition holds:
$$
\beta_*(\chi_X)=\gamma^*(\chi_{X'})\,.
$$
\begin{Remark}{\em
The set $\overline{\XExt}_{(C,B)}$, with $B$ a $C$-module, underlies a well known structure, called Baer sum, of abelian group (which turns out to be isomorphic to the third cohomology group of $C$ with coefficients in $B$, see \cite{Ger}). Changes of base determine group homomorphisms between such abelian groups.
}\end{Remark}

\subsection{Crossed $n$-fold extensions of groups} \label{sec:groups_n}

The interpretation of cohomology groups in terms of crossed extensions was extended from $n=2$ to all positive integers $n$ by Holt in \cite{Hol}, and independently by Huebschmann in \cite{Hue}. The crossed extensions they adopted are $n$-fold. We recall here their definition for the reader's convenience.

A crossed $n$-fold extension of groups is an exact sequence in $\Gp$
\begin{equation}
X\colon\qquad
\xymatrix@C=4ex{
0\ar[r]
& B\ar[r]^j
&G_n\ar[r]^{d_{n-1}}
&\cdots\ar[r]^{d_2}
& G_2\ar[r]^{\partial}
& G_1\ar[r]^p
& C\ar[r]
& 1}
\end{equation}
where
\begin{itemize}
\item[-]  $\partial\colon G_2\to G_1$ is a crossed module,
\item[-]  $B, G_n,\dots, G_3$ are $C$-modules,
\item[-]  $j,d_{n-1},\dots, d_2$ are $C$-equivariant.
\end{itemize}
Crossed extension morphisms are obviously described, and the category $\XExt^n$ is defined.
\begin{Theorem}
The commutative diagram of categories and functors below is a fiberwise opfibration in $\Fib(\Gp)$:
\begin{equation}
\begin{aligned}
\xymatrix{\XExt^n\ar[rr]^\Pi\ar[dr]_{\Pi_0}&&\Mod\ar[dl]^{(\underline{\ })_0}\\&\Gp}
\end{aligned}
\end{equation}
where $\Pi$ is the forgetful functor that sends a crossed $n$-fold extension $X$ as above to the $C$-module $B$, and $(\underline{\ })_0$ is the functor that assigns the group $C$ to the $C$-module $B$.
\end{Theorem}
\begin{proof}
The (sketch of the) proof is an obvious variation of the proof for the case of crossed extensions. Alternatively, it can be obtained from the more general case treated in Subsection \ref{sec:Moore}.
\end{proof}

Notice that, as for crossed extensions, also in this general case condition (C) for $\Pi$ holds for any $n$, so that $\Pi$ is actually an opfibration in $\Fib(\Gp)$.

Similarity relations can be set up just in the same way as we did for crossed extensions. Namely, we consider the equivalence relation generated by the morphisms
$$
(1,f_1,f_2,\dots, f_n,1)\,.
$$
Then, the category of the fibers of $\Pi$ is denoted by $\overline\XExt^n$
(notice that, $\XExt^2=\XExt$ and $\overline\XExt^2=\overline\XExt$).
Moreover, the constructions performed for crossed extensions can be easily adapted to crossed $n$-fold extensions, as for instance the three-fold factorization of a morphism and the description of morphisms in $\overline\XExt^n$ we just saw. 

Finally, we shall not analyze the degenerate case $\XExt^1=\mathbf{Opext}$. This reduces to the well known theory of group extensions with abelian kernel. Their (op)fibrational features are analyzed in classic manuals of homological algebra (see, for instance \cite[Chapter IV]{Homology}, Exercises 1--6).

\begin{Remark}{\em
Recall that, fixed a $C$-module $B$, the set of similarity classes of crossed $n$-fold extensions over the $C$-module $B$ gives an interpretation of the $(n+1)$-th  cohomology group $H^{n+1}(C,B)$ (see \cite{Homology}).
In our context, such groups (or better their underlying sets) are recovered as the fibers of $\bar\Pi$ arising from the comprehensive factorization:
$$
\begin{aligned}
\xymatrix{
\XExt^n(\Gp)\ar[r]^Q\ar[dr]_{\Pi_0}
&\overline{\XExt}^n(\Gp)\ar[r]\ar[d]_{\bar\Pi_0}\ar[r]^{\bar\Pi}
&\Mod(\Gp)\ar[dl]^{(\underline{\ })_0}\\&\Gp}
\end{aligned}
$$
}\end{Remark}

\subsection{Crossed $n$-fold extensions in strongly semi-abelian categories} \label{sec:Moore}

In this section, we show that the example of crossed $n$-fold extensions of groups can be extended all at once to many other concrete categories of non-abelian algebraic structures, working in a semi-abelian context \cite{JMT}.

In fact, we require \cC\ to be \emph{strongly semi-abelian}. Let us recall that a category is strongly semi-abelian (see \cite{BB}) if it is pointed, finitely complete, Barr-exact \cite{Barr}, it has coproducts and it is strongly protomodular. In our context, strong protomodularity can be characterized as follows:
for any morphism of split short exact sequences $(1_B,\varphi,\psi)$
$$
\xymatrix{
X\ar[r]^{\kappa}\ar[d]_{\psi}
&A\ar@<+.5ex>[r]^{\alpha}\ar[d]_{\varphi}
&B\ar@<+.5ex>[l]\ar@{=}[d]
\\
X'\ar[r]_{\kappa'}
&A'\ar@<+.5ex>[r]^{\alpha'}
&B\ar@<+.5ex>[l]
}
$$
\begin{itemize}
\item[(p)] If $\psi$ is an isomorphism, then $\varphi$ is an isomorphism as well.
\item[(sp)] If $\psi$ is a normal monomorphism, then  $\varphi\cdot\kappa=\kappa'\cdot\psi$ is a normal monomorphism as well.
\end{itemize}
The notion of strongly semi-abelian category encompasses several varietal examples: for instance, all the \emph{categories of groups with operations} in the sense of \cite{Orzech,Porter}, such as the categories of groups, of rings, of associative algebras, Poisson algebras, Lie and Leibniz algebras over a field, etc.\ \cite{MM2010}.

Strongly semi-abelian categories are a convenient setting for working with \emph{internal actions}. We recall their definition from \cite{BJK}. For a fixed object $B$ in \cC, let us consider the category $\Pt_B(\cC)=1_B\downarrow(\cC\downarrow B)$ of points over $B$, whose objects are pairs
\[
\xymatrix{
A \ar@<.5ex>[r]^p & B \ar@<.5ex>[l]^s
}
\]
with $p\cdot s = 1_B$. The kernel functor $\Ker_B\colon \Pt_B(\cC)\to \cC$, sending any such object to the kernel of $p$, is monadic and  the corresponding  monad, denoted by $B\flat(-)$, is defined, for any object $X$ of \cC, by the kernel diagram:
\[
\xymatrix{
B \flat X \ar[r]^-{\kappa_{B,X}} & B+X \ar[r]^-{[1,0]} & B \,.
}
\]
The $B\flat(-)$-algebras are called \emph{internal $B$-actions} \cite{BJK} and the category of such algebras is denoted by $\mathbf{Act}(B,-)$. For an action $\xi \colon B \flat X \to X$, the semi-direct product $X \rtimes_{\xi} B$ of $X$ with $B$ via $\xi$ was introduced in \cite{BJ} as the split epimorphism corresponding to $\xi$ via the canonical comparison equivalence $\Xi\colon\Pt_B(\cC) \to \mathbf{Act}(B,-)$. It can be computed explicitly (see \cite{MM2010}) via the coequalizer:
\[
\xymatrix{
B \flat X \ar@<.5ex>[r]^-{\kappa_{B,X}} \ar@<-.5ex>[r]_-{\iota_X\cdot\xi}
& B+X \ar[r]^-{q_\xi} & X \rtimes_{\xi} B \,,
}
\]
where $\iota_X$ is the canonical injection of $X$ into $B+X$.
Examples of internal actions follow:
\begin{itemize}
    \item the \emph{trivial action} of $B$ on $X$ is given by  the composite
        \[ \tau \colon \xymatrix{B \flat X \ar[r]^-{\kappa_{B,X}} & B+X \ar[r]^-{[0,1]} & X} \,; \]
    \item the \emph{conjugation action} of $X$ is given by the composite
        \[ \chi \colon \xymatrix{X \flat X \ar[r]^-{\kappa_{X,X}} & X+X \ar[r]^-{[1,1]} & X} \,; \]
    \item for any action $\xi\colon B\flat X \to X$ and any morphism $f\colon A \to B$, the composite
        \[ f^*(\xi) \colon \xymatrix{A \flat X \ar[r]^-{f \flat 1_X} & B \flat X \ar[r]^-{\xi} & X} \]
    defines an action, called the \emph{pullback action} of $\xi$ along $f$ (indeed, the above composition amounts to a pullback via the canonical comparison $\Xi$).
\end{itemize}

Internal actions allow, in a semi-abelian category \cC, to define internal crossed modules, introduced in \cite{J03} as a category equivalent to $\Cat(\cC)=\Gpd(\cC)$. When our base category is strongly semi-abelian, they can be described as follows (see \cite{mf_vdl,MM2010}).
An internal \emph{crossed module}  $(\partial,\xi)$ in $\cC$ is a map $\partial$ together with an action $\xi$
$$ \xymatrix{
	G_1 \flat G_2 \ar[r]^-{\xi} & G_2 \ar[r]^-{\partial} & G_1
} $$
such that the following two diagrams commute:
$$ \xymatrix{
	G_1 \flat G_2 \ar[r]^-{\xi} \ar[d]_-{1\flat \partial} & G_2 \ar[d]^-{\partial} \\
	G_1 \flat G_1 \ar[r]_-{\chi} & G_1
}\qquad
 \xymatrix{
	G_2 \flat G_2 \ar[r]^-{\chi} \ar[d]_{\partial\flat 1} & G_2 \ar@{=}[d] \\
	G_1 \flat G_2 \ar[r]_-{\xi} & G_2
} $$

The semi-direct product construction explained above serves as a main tool to produce the equivalence between internal crossed modules and internal categories. Every crossed module $(\partial \colon G_2 \to G_1, \xi)$ yields an internal category, whose underlying graph is
\[
\xymatrix{
G_2 \rtimes G_1 \ar@<.8ex>[r]^-{\pi_2} \ar@<-.8ex>[r]_-{[\partial,1\rangle} & G_1 \ar[l]|-{i_2}
}
\]
where $\pi_2$ and $i_2$ are the canonical projection and injection respectively and $[\partial,1\rangle$ is the unique arrow making the following diagram commute (where $i_1=\ker\pi_2$):
\[
\xymatrix{
G_2 \ar[r]^-{i_1} \ar[dr]_{\partial} & G_2 \rtimes G_1 \ar[d]^(.4){[\partial,1\rangle} & G_1 \ar[l]_-{i_2} \ar@{=}[dl] \\
& G_1
}
\]
The left hand side commutative triangle provides the construction of an internal crossed module starting from an internal category.

A large part of the theory of crossed modules of groups can be carried on internally in the strongly semi-abelian context (see \cite{Rodelo, Met2017}). For instance, one can prove that the kernel $B$ of $\partial$ is abelian (in fact, central in $G_2$) and that the cokernel $C$ of $\partial$ acts internally on $B$, so that $\partial$ induces on $B$ a structure of $C$-module. Following these lines,  we can adapt the following definition from \cite[Definition 8.1]{Rodelo}.

\begin{Definition} \label{def:X}
A crossed $n$-fold extension is an exact sequence in $\cC$
\begin{equation} \label{diag:x_n_ext}
X\colon\qquad
\xymatrix@C=4ex{
0\ar[r]
& B\ar[r]^j
&G_n\ar[r]^{d_{n-1}}
&\cdots\ar[r]^{d_2}
& G_2\ar[r]^{\partial}
& G_1\ar[r]^p
& C\ar[r]
& 1}
\end{equation}
where
\begin{itemize}
\item[-]  $\partial\colon G_2\to G_1$ is an internal crossed module,
\item[-]  $B, G_n,\dots, G_3$ are  internal $C$-modules,
\item[-]  $j,d_{n-1},\dots, d_2$ are $C$-equivariant, i.e.\ morphisms of internal $C$-actions.
\end{itemize}
\end{Definition}
\begin{Remark}\label{rk:acute}{\em
By saying that $d_2$ is $C$-equivariant we make a little abuse of language, since $C$ does not act on $G_2$, but on the image of $d_2$.
Let us notice  that a related issue will appear also in the case of associative algebras, see Definition~\ref{def:x_ext_alg}.
}\end{Remark}
Morphisms of crossed $n$-fold extensions are defined mimicking the group case, and they organize in the category $\XExt^n(\cC)$. Moreover, we denote by $\XExt(\cC)$, $\Mod(\cC)$ the categories of crossed $2$-fold extensions in $\cC$ and that of $\cC$-modules, respectively.

In \cite[Proposition 8.6]{Rodelo}, the authoress defines, for an object $C$ of $\cC$,  an $n$-order direction functor
$$
d_{n,C}\colon \xymatrix{\XExt^n_{C}(\cC)\ar[r]&\Mod_C(\cC)}
$$
that assigns to a crossed $n$-fold extension $X$ as in (\ref{diag:x_n_ext}), the $C$-module structure on $B$ induced by $X$.  In fact, $d_{n,C}$ is nothing but the restriction to fibers of a functor $\Pi\colon \XExt^n(\cC)\to \Mod(\cC)$, which  is actually a 1-morphism in $\Fib(\cC)$.
\begin{Theorem}\label{thm:fib_opfib_Moore}
Let $\cC$ be strongly semi-abelian.
The commutative diagram of categories and functors below is a fiberwise opfibration in $\Fib(\cC)$:
\begin{equation}
\begin{aligned}\label{diag:str_triangle}
\xymatrix{\XExt^n(\cC)\ar[rr]^\Pi\ar[dr]_{\Pi_0}&&\Mod(\cC)\ar[dl]^{(\underline{\ })_0}\\&\cC}
\end{aligned}
\end{equation}
where $\Pi$ is the forgetful functor that sends a crossed $n$-fold extension $X$ as above to the $C$-module $B$, and $(\underline{\ })_0$ is the functor that assigns to the $C$-module $B$, the object $C$.
\end{Theorem}

\begin{proof}
We only sketch the proof, which follows the lines of the group theoretical case. The fact that $\Pi$ is a 1-morphism in $\Fib(\cC)$ is proved in the same way as in the proof of Theorem \ref{thm:fib_opfib_XModGp}, since also in $\XExt^n(\cC)$ cartesian morphisms are obtained by pullback.

The fact that, for any object $C$ in $\cC$, the restriction $\Pi_C=d_{n,C}$ to fibers is an opfibration can be deduced by the properties of the direction functor in Section 4.3 of \cite{Rodelo}. However, the constructions provided in [\emph{loc.cit.}] are not always explicit. Actually, a notion of push forward is available in semi-abelian categories (see \cite{pf} for a full treatment). This allows us to describe the opcartesian liftings by the same strategy as in the group case, by replacing group theoretical constructions with internal ones. For the reader's convenience, we provide an explicit and  self-contained proof for the case $n=2$ in the following lemma.
\end{proof}

\begin{Lemma}[Push forward of a crossed extension] \label{lemma:pf_xext}
Given a crossed extension
\[
X \colon \qquad
\xymatrix{
0 \ar[r] & B\ar[r]^j & G_2 \ar[r]^{\partial} & G_1 \ar[r]^p & C\ar[r] & 1
}
\]
and a $C$-module morphism $\beta\colon B \to B'$, there exists an opcartesian lifting of $\beta$ at $X$. In other words, the functor $d_{2,C}\colon\XExt^2_C(\cC)\to\Mod_C(\cC)$ is an opfibration.

We denote by $\beta_*X$ the codomain of this opcartesian morphism and we call it the \emph{push forward} of $X$ along $\beta$.
\end{Lemma}

\begin{proof}
Let us fix some notations.
\begin{itemize}
 \item $\xi\colon G_2\flat G_1 \to G_1$ is the action endowing $\partial$ of a crossed module structure;
 \item $\phi\colon C\flat B \to B$ is the action of $C$ on $B$ induced by the crossed module $(\partial,\xi)$ and making $B$ a $C$-module;
 \item $\phi'\colon C\flat B' \to B'$ is the action endowing $B'$ with a $C$-module structure.
\end{itemize}
The push forward of $X$ along $\beta$ can be obtained by means of the push forward construction provided in \cite{pf}, applied to the normal monomorphism $j$ and the $C$-module morphism $\beta$. Let us first observe that $j$ is a central monomorphism, being the kernel of a crossed module, hence the conjugation action of $G_2$ on $B$ is trivial. If we also regard $B'$ as a trivial $G_2$-module, then $\beta$ is clearly equivariant with respect to the $G_2$-actions and the semidirect product $B'\rtimes G_2$ is in fact a direct product. It is immediate then to see that the conditions for the push forward are fulfilled and we can describe explicitly the construction following the lines of Section 2.3 in \cite{pf}.

The precrossed module of Lemma 2.11 in \cite{pf} reduces, in our case, to the morphism $\langle \beta,-j\rangle \colon B \to B'\times G_2$, which is monomorphic (since its second projection $-j$ is), hence it is a kernel (notice that the existence of $-j$ is ensured by the centrality of $B$ in $G_2$). So, denoting by $q$ its cokernel, we have a short exact sequence
\[
\xymatrix{
B \ar[r]^-{\langle\beta,-j\rangle} & B'\times G_2 \ar[r]^-q & B'\times^B G_2\,.
}
\]
The composite $j'=q\cdot\langle 1,0\rangle\colon B'\to B'\times^B G_2$ provides the push forward of $j$ along $\beta$:
\[
\xymatrix{
B \ar[r]^-j \ar[d]_{\beta} \ar[dr]^{\langle\beta,-j\rangle} & G_2 \ar[d]^{\langle 0,1\rangle} \\
B' \ar[r]_-{\langle 1,0\rangle} \ar@{=}[d] & B'\times G_2 \ar[d]^q \\
B' \ar[r]_-{j'} & B'\times^B G_2
}
\]
As a consequence of Theorem 2.13 in \cite{pf}, $j'$ is a normal monomorphism and the morphism $(\beta,q\cdot\langle 0,1\rangle)$ of crossed modules induces an isomorphism on cokernels.

We can now compose the cokernel of $j'$ with the monomorphic part in the (regular epi, mono) factorization of $\partial$, to obtain a morphism $\partial'\colon B'\times^B G_2\to G_1$. It is immediate to observe that the following is then an exact sequence:
\[
\xymatrix{
0 \ar[r] & B \ar[r]^-{j'} & B'\times^B G_2 \ar[r]^-{\partial'} & G_1 \ar[r]^-p & C \ar[r] & 1\,.
}
\]
It remains to prove that $\partial'$ is a crossed module.

First of all, let us notice that $G_1$ acts on $B$ with the action
\[
\xymatrix{
p^*(\phi)\colon G_1\flat B \ar[r]^-{p\flat 1} & C\flat B \ar[r]^-{\phi} & B\,.
}
\]
Moreover, it acts on $B'$ with $p^*(\phi')$ and on $G_2$ with $\xi$, hence a compatible action of $G_1$ on the product $B'\times G_2$ is also defined (corresponding to the product in $\Pt_{G_1}(\cC)$). Let us denote the latter by $\psi$. It turns out that $\langle \beta,-j\rangle$ is equivariant with respect to the $G_1$-actions, i.e.\ the following diagram commutes:
\[
\xymatrix{
G_1\flat B \ar[r]^-{1\flat\langle \beta,-j\rangle} \ar[d]_{p^*(\phi)} & G_1\flat(B'\times G_2) \ar[d]^{\psi} \\
B \ar[r]^-{\langle \beta,-j\rangle} & B'\times G_2
}
\]
To prove this, it suffices to compose with product projections and verify that $\beta$ and $-j$ are both equivariant. Projecting on the first component, we get the diagram
\[
\xymatrix{
G_1\flat B \ar[r]^-{1\flat\beta} \ar[d]_{p\flat 1} & G_1\flat B' \ar[d]^{p\flat 1} \\
C\flat B \ar[r]^-{1\flat\beta} \ar[d]_{\phi} & C\flat B' \ar[d]^{\phi'} \\
B \ar[r]^-{\beta} & B'
}
\]
which commutes by equivariance of $\beta$. Projecting on the second component we get the following:
\begin{equation} \label{diag:-j}
\begin{aligned}
\xymatrix{
G_1\flat B \ar[r]^-{1\flat(-j)} \ar[d]_{p^*(\phi)} & G_1\flat G_2 \ar[d]^{\xi} \\
B \ar[r]^-{-j} & G_2
}
\end{aligned}
\end{equation}
We prove it commutes by means of the equivalences between actions and points, groupoids and crossed modules. Recall that the action $\phi$ is induced by the crossed module structure on $\partial$, i.e.\ it is nothing but the action induced by the \emph{direction} of the groupoid corresponding to $\partial$, whose construction is displayed below:
\[
\xymatrix@C=7ex{
B \rtimes(G_2\rtimes G_1) \ar@<1ex>[r]^-{1\rtimes \pi_2} \ar@<-1ex>[r]_-{1\rtimes[\partial,1\rangle} \ar@<1ex>[d]^{\pi_2} \ar@<-1ex>[d]_{[i_1j,1\rangle} & B\rtimes G_1 \ar@<1ex>[d]^{\pi_2} \ar@<-1ex>[d]_{\pi_2} \ar[r]^-{1\rtimes p} \ar[l]|-{1 \rtimes i_2} & B\rtimes C \ar@<1ex>[d]^{\pi_2} \ar@<-1ex>[d]_{\pi_2} \\
G_2 \rtimes G_1 \ar@<1ex>[r]^-{\pi_2} \ar@<-1ex>[r]_-{[\partial,1\rangle} \ar[u]|{i_2} & G_1 \ar[r]_-p \ar[l]|-{i_2} \ar[u]|{i_2} & C \ar[u]|{i_2}
}
\]
The leftmost vertical reflexive pair of arrows is in fact the kernel relation of the map $\langle [\partial,1\rangle,\pi_2\rangle\colon G_2\rtimes G_1\to G_1\times G_1$ induced by the projections of the groupoid $G_2\rtimes G_1$ over $G_1$. In other words, it is the equivalence relation associated with the normal subobject $\ker(d)\cap\ker(c)$ of $G_2\rtimes G_1$, which is nothing but $B$. The rightmost vertical reflexive pair is the direction of this groupoid, corresponding to the $C$-module $B$. It is easy to see now that the following is a morphism of points over $G_1$:
\[
\xymatrix@C=7ex{
B\rtimes G_1 \ar@<.5ex>[d]^{\pi_2} \ar[r]^-{[i_1j,1\rangle(1\rtimes i_2)} & G_2 \rtimes G_1 \ar@<.5ex>[d]^{\pi_2} \\
G_1 \ar@<.5ex>[u]^{i_2} \ar@{=}[r] & G_1 \ar@<.5ex>[u]^{i_2}
}
\]
Indeed, $[i_1j,1\rangle(1\rtimes i_2)i_2=i_2\pi_2 i_2=i_2\colon G_1\to G_2\rtimes G_1$ and $\pi_2[i_1j,1\rangle(1\rtimes i_2)=\pi_2(1\rtimes\pi_2)(1\rtimes i_2)=\pi_2\colon B\rtimes G_1 \to G_1$. Moreover, the induced restriction to the kernels of the split epimorphisms is $j$, since $[i_1j,1\rangle(1\rtimes i_2)i_1=[i_1j,1\rangle i_1=i_1j\colon B\to G_2\rtimes G_1$. Equivalently, we have a morphism of actions
\[
\xymatrix{
G_1\flat B \ar[r]^-{1\flat j} \ar[d]_{p^*(\phi)} & G_1\flat G_2 \ar[d]^{\xi} \\
B \ar[r]^-{j} & G_2
}
\]
yielding the desired (\ref{diag:-j}) by composition with the inversion map of $B$.

So far, we have proved that $\langle \beta,-j\rangle$ is equivariant with respect to the $G_1$-actions. Hence, by Theorem 5.5 in \cite{Met2017}, the quotient is also equivariant and $B'\times^B G_2$ is endowed with a $G_1$-action, that we denote by $\xi'$. We are going to see that this action gives rise to a crossed module structure on $\partial'$. Before, let us observe that $\partial'q=\partial\pi_2\colon B'\times G_2 \to G_1$, this equality being easily proved by composition with the jointly epimorphic pair of inclusions in the product $B'\times G_2$.

In order to prove that $(\partial',\xi')$ is a crossed module, we have to check the validity of precrossed module and Peiffer conditions, i.e.\ the commutativity of the following diagrams, where $\chi$ denotes conjugation action in both cases:
{\small
\[
\xymatrix{
G_1\flat(B'\times^B G_2) \ar[d]_{\xi'} \ar[r]^-{1\flat\partial'} \ar@{}[dr]|{(\mathrm{PX})} & G_1\flat G_1 \ar[d]^{\chi} & (B'\times^B G_2)\flat(B'\times^B G_2) \ar[d]_{\chi} \ar[r]^-{\partial'\flat 1} \ar@{}[dr]|{(\mathrm{PEI})} & G_1\flat(B'\times^B G_2) \ar[d]^{\xi'} \\
B'\times^B G_2 \ar[r]_-{\partial'} & G_1 & B'\times^B G_2 \ar@{=}[r] & B'\times^B G_2
}
\]
}

For the first one, we can compose with $1\flat q\colon G_1\flat(B'\times G_2)\to G_1\flat(B'\times^B G_2)$, which is a regular epimorphism since $q$ is. Then, recalling that $\partial'q=\partial\pi_2$, we get the following diagram
\[
\xymatrix{
G_1\flat(B'\times G_2) \ar[d]_{\psi} \ar[r]^-{1\flat\pi_2} \ar@{}[dr]|{(a)} & G_1\flat G_2 \ar[d]^{\xi} \ar[r]^-{1\flat\partial} \ar@{}[dr]|{(b)} & G_1\flat G_1 \ar[d]^{\chi} \\
B'\times G_2 \ar[r]_-{\pi_2} & G_2 \ar[r]_-\partial & G_1
}
\]
where $(a)$ commutes by definition of the action $\psi$ and the commutativity of $(b)$ is nothing but the precrossed module condition on $\partial$. So $(\mathrm{PX})$ commutes by cancellation.

For the second one, we compose with $q\flat q$ which is again a regular epimorphism. Then, again by means of the equality $\partial'q=\partial\pi_2$, we get the following:
\[
\xymatrix{
(B'\times G_2)\flat(B'\times G_2) \ar[d]_{\chi} \ar[r]^-{\pi_2\flat 1} \ar@{}[dr]|{(c)} &  G_2\flat(B'\times G_2) \ar[r]^-{\partial\flat 1} \ar[d]_{\langle 0,1\rangle^*\chi} \ar@{}[dr]|{(d)} & G_1\flat(B'\times G_2) \ar[d]_{\psi} \ar[r]^-{1\flat q} \ar@{}[dr]|{(e)} & G_1\flat(B'\times^B G_2) \ar[d]^{\xi'} \\
B'\times G_2 \ar@{=}[r] & B'\times G_2 \ar@{=}[r] & B'\times G_2 \ar[r]_-q & B'\times^B G_2
}
\]
where $(e)$ commutes by definition of the action $\xi'$, and $(c)$ commutes essentially by abelianness of $B'$. We finally prove the commutativity of $(d)$ by composing with product projections.

Projecting on the first component we get the following diagram:
\[
\xymatrix{
G_2\flat(B'\times G_2) \ar[r]^-{\partial\flat 1} \ar[d]_{\langle 0,1\rangle^*\chi} \ar@{}[dr]|{(d)} & G_1\flat(B'\times G_2) \ar[d]_{\psi} \ar[r]^-{1\flat \pi_1} \ar@{}[dr]|{(f)} & G_1\flat B' \ar[d]_{p^*\phi'} \ar[r]^-{p\flat 1} \ar@{}[dr]|{(g)} & C\flat B' \ar[d]^{\phi'} \\
B'\times G_2 \ar@{=}[r] & B'\times G_2 \ar[r]_-{\pi_1} & B' \ar@{=}[r] & B'
}
\]
which commutes if and only if the following commutes:
\[
\xymatrix{
G_2\flat(B'\times G_2) \ar[d]_{\langle 0,1\rangle^*\chi} \ar[r]^-{1\flat \pi_1} \ar@{}[dr]|{(h)} & G_2\flat B' \ar[d]_{\tau} \ar[r]^-{0\flat 1} \ar@{}[dr]|{(i)} & C\flat B' \ar[d]^{\phi'} \\
B'\times G_2 \ar[r]_-{\pi_1} & B' \ar@{=}[r] & B'
}
\]
where $\tau$ denotes the trivial action. But in the latter $(h)$ commutes since the conjugation action of one factor on another in a direct product is trivial, and $(i)$ commutes because $0^*(-)$ of any action is also trivial.

Projecting on the second component we get the following diagram
\[
\xymatrix{
G_2\flat(B'\times G_2) \ar[r]^-{\langle 0,1\rangle\flat 1} \ar[d]_{\langle 0,1\rangle^*\chi} & (B'\times G_2)\flat(B'\times G_2) \ar[d]_{\chi} \ar[r]^-{\pi_2\flat \pi_2} & G_2\flat G_2 \ar[d]_{\chi} \ar[r]^-{\partial\flat 1}  & G_1\flat G_2 \ar[d]^{\xi} \\
B'\times G_2 \ar@{=}[r] & B'\times G_2 \ar[r]_-{\pi_2} & G_2 \ar@{=}[r] & G_2
}
\]
where the left square commutes by definition, the middle one commutes because morphisms are equivariant with respect to conjugation actions, and the right square is nothing but the Peiffer condition on $\partial$.
\end{proof}

\begin{Remark}{\em
Let us observe that the statement of Lemma \ref{lemma:pf_xext} can be easily deduced from the more general result proved by Bourn in \cite[Proposition 10]{Bourn02} for internal groupoids. On the other hand, the proof provided in [{\em loc.cit.}], as well as its normalization given in \cite{Rodelo}, is hard to unfold if one wants to describe explicitly the opcartesian liftings. However, the main point of Lemma \ref{lemma:pf_xext} (and of its proof) is to give a recipe to determine the opcartesian liftings by means of a construction that can be effectively performed in several  categories of algebraic varieties, namely the \emph{push forward} referred above. For instance, in \cite{pf} we show how to construct the push forward in the category of rings, in the category of Lie algebras and in the category of  Leibniz algebras. In the next example, the case of associative algebras is analyzed.
}\end{Remark}

We shall not consider the degenerate case $n=1$ of extensions with abelian kernels. The proof of the opcartesian aspects can be deduced from \cite{BournJanelidze}. A proof that uses explicitly the push forward can be found in \cite[Proposition 5.1]{pf}.

As a matter of fact, likewise in the case of group extensions considered in Sections \ref{sec:groups} and \ref{sec:groups_n}, also in the strongly semi-abelian context the fibrewise opfibration $\Pi$ above enjoys condition (C), so that the following statement holds.

\begin{Proposition} \label{prop:Xext_opfib}
The morphism $\Pi$ in diagram (\ref{diag:str_triangle}) is an opfibration in $\Fib(\cC)$.
\end{Proposition}

Finally, Theorem \ref{thm:fib_opfib_Moore} above allows us to gather  in a unified framework various cohomology theories present in the literature, as the following remark explains.
\begin{Remark}{\em
Recall that, fixed a $C$-module $B$, the set of equivalence classes of crossed $n$-fold extensions over the $C$-module $B$ gives an interpretation of the $n$-th Bourn cohomology group $H^n(C,B)$ (\cite{bourn}, see also \cite{Rodelo} for the special case of Moore categories).
In our context, such groups (or better, their underlying sets) are recovered as the fibers of $\bar\Pi$ in the comprehensive factorization:
$$
\begin{aligned}
\xymatrix{
\XExt^n(\cV)\ar[r]^Q\ar[dr]_{\Pi_0}
&\overline{\XExt}^n(\cC)\ar[r]\ar[d]_{\bar\Pi_0}\ar[r]^{\bar\Pi}
&\Mod(\cC)\ar[dl]^{(\underline{\ })_0}\\&\cC}
\end{aligned}
$$
}\end{Remark}

\subsection{Case study: associative algebras and Hochschild cohomology} \label{sec:AssAlg}

Let $k$ be a field. The category $\AssAlg=k\text{-}\AssAlg$ of associative algebras over $k$ is a category of groups with operations (see \cite{Porter}). Hence, it is strongly semi-abelian, so that one can consider the specialization of diagram (\ref{diag:str_triangle}) to the fiberwise opfibration
\begin{equation}\label{diag:fib_opfib_AssAlg}
\begin{aligned}
\xymatrix{\XExt^n(\AssAlg)\ar[rr]^\Pi\ar[dr]_{\Pi_0}&&\Mod(\AssAlg)\ar[dl]^{(\underline{\ })_0}\\&\AssAlg}
\end{aligned}
\end{equation}
Different cohomology theories for associative algebras have been introduced by many authors, see for instance \cite{BauPi06}. In this section we briefly recall a notion of crossed $n$-fold extension that has been used for an interpretation of Hochschild cohomology of \emph{unital} associative algebras in \cite{BauMi02}. Then, we show how such a theory can be described by a fiberwise opfibered setting, and how this is related with the general setting of diagram  (\ref{diag:fib_opfib_AssAlg}).

\smallskip

Unital associative algebras (over $k$) are associative algebras endowed with multiplicative unit, and their morphisms are algebra morphisms that preserve such units. If we denote the category of unital associative algebras and their morphisms by $\AssAlg_1$, there is a non-full inclusion of categories
$$
\AssAlg_1\to\AssAlg\,.
$$
Since $\AssAlg$ is a category of groups with operations, we can make explicit the definition of a $B$-action on $A$. From \cite {Porter},  a $B$-action on $A$ ($B$-\emph{structure} in [{\em loc.cit.}]) is a pair of bilinear operations $B\times A\to A$ and $A\times B\to A$ (both denoted by $*$ as well as the algebra multiplication) satisfying:
\begin{itemize}
\item[(S1)] $(a*a')*b=a*(a'*b)\qquad b*(a*a')=(b*a)*a'$
\item[(S2)] $a*(b*b')=(a*b)*b'\qquad (b*b')*a=b*(b'*a)$
\item[(S3)] $(b*a)*b'=b*(a*b')\qquad a*(b*a')=(a*b)*a'$
\end{itemize}
for all $a,a'\in A$ and $b,b'\in B$.
As explained in \cite{Porter}, a $B$-action on $A$ corresponds to a split epimorphism over $B$ with kernel $A$.

On the other hand, although $\AssAlg_1$ is not a category of groups with operations (since it has two constants  $0\neq 1$) still it  makes sense to speak of $B$-actions for $B$ a unital algebra, but in this case, $B$ would act on a non-unital algebra $A$. Let us consider a split epimorphism with its section
$$
\xymatrix{
E\ar@<+.5ex>[r]^{p}&B\ar@<+.5ex>[l]^{s}
}
$$
in $\AssAlg_1$ and consider it as a split epimorphism in $\AssAlg$. This corresponds to a $B$-action on the ideal $A=\Ker(p)$, which is not, in general, an object of $\AssAlg_1$. Notice that, in  this case, the two bilinear operations corresponding to this kind of actions satisfy the following additional axioms:
\begin{itemize}
\item[(S4)] $1*a=a\,, \qquad a*1=a$\,.
\end{itemize}
We shall call such actions \emph{unital actions}.

Let $C$ be a unital associative  algebra. A \emph{bimodule} $M$ over $C$ is a $k$-vector space endowed with left and right unital $C$-actions such that for $c, c' \in C$ and $m \in  M$ we have $(c*m)*c' = c*(m*c')$.  Morphisms of bimodules are obviously defined, and they give rise to a category denoted by $\Bimod$.  We can state the following lemma (where a $0$-algebra is an algebra where the multiplication is trivial).

\begin{Lemma}\label{lm:bimod_to_mod}
If $M$ is a bimodule over a unital associative algebra $C$, then te corresponding $C$-action gives rise to a unital $C$-action on $M$, considered as a $0$-algebra in $\AssAlg$. As a consequence, we can identify the category $\Bimod$ of bimodules  with a non full subcategory of $\Mod(\AssAlg)$.
\end{Lemma}
\begin{proof}
The proof follows from the description of the unital actions above and from the fact that a $C$-action on a $0$-algebra has both equations in (S1) and the second one  of (S3) identically zero,  while both equations in (S2) and the first one of (S3) give right-, left- and bi- module conditions respectively.  The last statement follows from the fact that abelian objects in $\AssAlg$ are precisely $k$-vector spaces endowed with a trivial multiplication.
\end{proof}
Let us remark that, in general, a $C$-action on a $0$-algebra $M$ does not give a bimodule, unless the action is unital.

Let $B$ be a unital associative algebra. A \emph{crossed bimodule} is a morphism of $B$-bimodules
$$
\partial\colon A\to B\,,
$$
(where $B$ is considered as a $B$-bimodule via the multiplication in $B$) satisfying  $$\partial(a)*a'=a*\partial(a')\,,$$ for all $a,a'\in A$.
Crossed bimodules are called just crossed modules in \cite[Definition 5.9]{BauMi02}.
It is possible to show (see \cite{BauPi06}) that the product defined by $a_1*a_2=\partial(a_1)*a_2$ gives $A$ a structure of a not (necessarily) unital associative algebra, in such a way that $\partial$ is a morphism in $\AssAlg$.
As a matter of fact, any crossed bimodule $\partial$ as above determines a crossed module in $\AssAlg$. This is easily proved by specializing the description of crossed modules in \cite{Porter} to the case of associative algebras as follows. A morphism of associative algebras $\partial\colon A\to B$, together with a $B$-action on $A$ is a crossed module if it satisfies (compare with conditions (iii) and (iv) in \cite[Proposition 2]{Porter}):
\begin{itemize}
\item[(XM1)] \ $\partial(a_1)*a_2=a_1*a_2=a_1*\partial(a_2)$\,,
\item[(XM2)] \ $\partial(b*a)=b*\partial(a)$ and\/  $\partial(a*b)=\partial(a)*b$\,.
\end{itemize}
Indeed, given a crossed bimodule, conditions (XM2) follow from left/right-module axioms, and conditions (XM1) come from the definition of the multiplication on the $B$-bimodule $A$ given above.

Let us denote by $\XBimod$ the category of crossed bimodules of associative unital algebras, where a morphism between two crossed bimodules
$$\partial\colon A\to B\qquad \text{and}\qquad   \partial'\colon A'\to B'\,,$$ is a morphism  $(\beta,\alpha)$ of bimodules such that $\beta\colon B\to B'$ is also a morphism of unital algebras. By the discussion above, $\XBimod$ can be identified with a non full subcategory of $\XMod(\AssAlg)$.

Our discussion plainly extends to crossed extensions. Indeed, given a crossed bimodule $\partial \colon A\to B$, one can form the following exact sequence of $k$-vector spaces
\begin{equation}\label{diag:XExtAssAlg}
\xymatrix{
0\ar[r]
&M\ar[r]^j
&A\ar[r]^\partial
&B\ar[r]^p
&C\ar[r]
&0
}
\end{equation}
and prove the following facts \cite[Lemma 3.2.1]{BauPi06}:
\begin{itemize}
\item $C$ has a natural structure of associative algebra with unit, that makes $p$ a morphism in $\AssAlg_1$,
\item the image of $\partial$ is an ideal of $B$, i.e.\ the kernel of $p$ in $\AssAlg$,
\item $M*A=0=A*M$, i.e.\ $M$, considered as a subalgebra of $A$ in $\AssAlg$, is central in $A$,
\item the action of $B$ on $A$ determines a well-defined unital action of $C$ on $M$, making $M$ a $C$-bimodule.
\end{itemize}
We are led to the following definition.
\begin{Definition}\label{def:XExtAssAlg}
Given a unital associative algebra $C$ and a $C$-bimodule $M$, a crossed biextension of $C$ by $M$ is an exact sequence
$$\xymatrix{
0\ar[r]
&M\ar[r]^j
&A\ar[r]^\partial
&B\ar[r]^p
&C\ar[r]
&0
}
$$
in $k$-$\Vect$, where $\partial$ is a crossed bimodule, $p$ is a morphism of unital algebras and the action of $C$ on $M$ determined by the crossed bimodule structure of $\partial$ coincides with the given $C$-bimodule structure on $M$.

Morphisms of crossed biextensions are obviously defined, and the category of crossed biextensions will be denoted by $\XBiext$.
\end{Definition}

Let us observe that, since any crossed bimodule determines a crossed module in $\AssAlg$ as described above and the exactness of a sequence in $\AssAlg$ can be checked in $k$-$\Vect$, any crossed biextension can be seen also as a crossed extension in $\AssAlg$.

Notice that in  \cite{BauPi06} crossed biextensions are called  \emph{crossed extensions}, but here we prefer to avoid using the same terminology since we want to keep it for crossed extensions of not (necessarily) unital algebras.

\begin{Theorem} \label{thm:fib_opfib_AssAlg}
The fiberwise opfibration in $\Fib(\AssAlg)$ determined by Theorem \ref{thm:fib_opfib_Moore}
\begin{equation}\label{diag:up}
\begin{aligned}
\xymatrix{\XExt(\AssAlg)\ar[rr]^\Pi\ar[dr]_{\Pi_0}&&\Mod(\AssAlg)\ar[dl]^{(\underline{\ })_0}\\&\AssAlg}
\end{aligned}
\end{equation}
restricts to a fiberwise opfibration in $\Fib(\AssAlg_1)$:
\begin{equation}\label{diag:down}
\begin{aligned}
\xymatrix{\XBiext\ar[rr]^\Pi\ar[dr]_{\Pi_0}&&\Bimod\ar[dl]^{(\underline{\ })_0}\\&\AssAlg_1}
\end{aligned}
\end{equation}
\end{Theorem}
\begin{proof}
By the identifications described in the discussion above, diagram  (\ref{diag:down}) is consistently a restriction of diagram (\ref{diag:up}). Moreover, it is easy to show that it is a diagram in  $\Fib(\AssAlg_1)$, since the $\Pi_0$-cartesian morphisms in $\XBiext$ are precisely those morphisms of $\XBiext$ which are $\Pi_0$-cartesian in $\XExt(\AssAlg)$, i.e.\ those of the form:
$$
\xymatrix@!C=6ex{
0\ar[r]
&M\ar[r]\ar@{=}[d]
&A\ar[r]\ar@{=}[d]
&B\times_CC'\ar[r]\ar[d]\ar@{}[dr]|{\text{p.b.}}
&C'\ar[r]\ar[d]^{\gamma}
&0
\\
0\ar[r]
&M\ar[r]_j
&A\ar[r]_{\partial}
&B\ar[r]_p
&C\ar[r]
&0}
$$
with $p$ and $\gamma$ in $\AssAlg_1$. The rest of the proof is straightforward.

Furthermore it is a fiberwise opfibration, as one could deduce from \cite[Proposition 4.6]{BauMi02}. However, it is worth analyzing the construction of opcartesian liftings, in order to fill some  missing details in the proof provided in [\emph{loc.cit.}].
Let us consider a crossed biextension $(p,\partial,j)$ as above, together with a $C$-bimodule morphism $\mu\colon M\to M'$. First we compute the push out of $j$ along $\mu$ in the category of $k$-vector spaces, thus obtaining the following morphism of exact sequences in $k\text{-}\mathbf{Vect}$:
$$
\xymatrix@!C=6ex{
0\ar[r]
&M\ar[r]^j\ar[d]_{\mu}\ar@{}[dr]|{\text{p.o.}}
&A\ar[r]^\partial\ar[d]^\alpha
&B\ar[r]^p\ar@{=}[d]
&C\ar[r]\ar@{=}[d]
&0
\\
0\ar[r]
&M'\ar[r]_-{j'}
&M'+_MA\ar[r]_-{\partial'}
&B\ar[r]_p
&C\ar[r]
&0}
$$
Now, since $k\text{-}\mathbf{Vect}$ is an abelian category, the vector space $M'+_MA$ can be described as the space of equivalence classes of pairs $(m',a)\in M'\times A$, with $ (m'_2,a_2)\in[m'_1,a_1]$ if $a_2-a_1\in M$ and $\mu(a_2-a_1)=m'_2-m'_1$. In fact, $M'+_MA$ is a $B$-bimodule. The left/right module structure is well-defined by:
$$
b*[m',a]=[p(b)*m',b*a]\,,\qquad [m',a]*b=[m'*p(b),a*b]\,.
$$
Indeed, for $(m'_2,a_2)\in[m'_1,a_1]$ and any $b\in B$, one has
$$
b*[m'_2,a_2]\eq{def}[p(b)*m'_2,b*a_2]=[p(b)*(m'_2-m'_1),b*(a_2-a_1)]+[p(b)*m'_1,b*a_1]=
$$
$$
=[p(b)*\mu(a_2-a_1),b*(a_2-a_1)]+[p(b)*m'_1,b*a_1]=
$$
$$
=[\mu(p(b)*(a_2-a_1)),b*(a_2-a_1)]+[p(b)*m'_1,b*a_1]=
$$
$$
=[0,0]+[p(b)*m'_1,b*a_1]=[p(b)*m'_1,b*a_1]\eq{def}b*[m'_1,a_1]
$$
where in the fourth equality we have used the equivariance of $\mu$.
Similarly for the right action. Moreover,
$$
b_1*([m',a]*b_2)=b_1*[m'*p(b_2),a*b_2]=[p(b_1)*(m'*p(b_2)),b_1*(a*b_2)]=
$$
$$
=[(p(b_1)*m')*p(b_2),(b_1*a)*b_2)]=[p(b_1)*m',b_1*a]*b_2=(b_1*[m',a])*b_2
$$
where in the third equality we have used that $A$ is a $B$-bimodule and $M'$ is a $C$-bimodule.

The $k$-linear map $\partial'$ is determined by the universal property of coproducts in $k$-$\mathbf{Vect}$: $\partial'([m',a])=\partial(a)$.
In fact, $\partial'$ is a crossed bimodule, as the following calculations show. It is a morphism of $B$-bimodules:
$$
\partial'(b*[m',a])=\partial'([p(b)*m',b*a])=
\partial(b*a)=b*\partial(a)=b*\partial'([m',a])
$$
Moreover,
$$
\partial'([m'_1,a_1])*[m'_2,a_2]=\partial(a_1)*[m'_2,a_2]=[(p\partial)(a_1)*m'_2,\partial(a_1)*a_2]=
$$
$$
=[0*m'_2,\partial(a_1)*a_2]=[0,\partial(a_1)*a_2]=[0,a_1*\partial(a_2)]=[m'_1*0,a_1*\partial(a_2)]=
$$
$$
=[m'_1*(p\partial)(a_2),a_1*\partial(a_2)]=[m'_1,a_1]*\partial(a_2)=[m'_1,a_1]*\partial'([m'_2,a_2])
$$
where the central equality follows from the crossed bimodule condition on $\partial$.

Now it is routine to check that $(p,\partial',j')$ is a crossed biextension, and that $(1_C,1_B,\alpha,\mu)$ is a morphism of crossed biextensions.

Moreover, $(1_C,1_B,\alpha,\mu)$ is opcartesian over $\mu$ in $\XBiext$. This can be checked directly, or by the fact that it is opcartesian also in $\XExt(\AssAlg)$. With reference to the general case of strongly semi-abelian categories, one observes that the pushout computed in $k$-$\mathbf{Vect}$ in the above diagram, turns out to be a push forward when the diagram is considered in  $\AssAlg$. This can be proved by means of the characterization of push forwards given in \cite[Theorem 2.13]{pf}: the square $\alpha\cdot j=j'\cdot \mu$ presents $j'$ as the push forward of $j$ along $\mu$
since $(\alpha,\mu)$ induces an isomorphism between the cokernels of $j$ and $j'$ (both are kernels of $p$) and a regular epimorphism between the kernels (in fact, the last is an isomorphism too, since $j$ and $j'$ are monomorphisms).
\end{proof}

As a matter of fact, \cite{BauMi02} treats the more general case of crossed $n$-fold biextensions. It actually fits in our general theory as well as the case of crossed biextensions we have just seen, i.e. with $n=2$.

Here we only quote the definition \cite[Definition 4.1]{BauMi02} (compare with Definition \ref{def:X}) and state the following theorem, the proof being essentially the same as in the case $n=2$.

\begin{Definition}\label{def:x_ext_alg}
Given an associative unital algebra $C$, and a $C$-bimodule $M$, a crossed $n$-fold biextension of $C$ by $M$ is an exact sequence in $k$-Vect:
\begin{equation}\label{diag:x_n_biext}
X\colon\qquad\xymatrix@C=4ex{
0\ar[r]
& M\ar[r]^j
&A_n\ar[r]^{d_{n-1}}
&\cdots\ar[r]^{d_2}
& A_2\ar[r]^{\partial}
& B\ar[r]^p
& C\ar[r]
& 0}
\end{equation}
where
\begin{itemize}
\item[-]  $\partial\colon A_2\to B$ is a crossed bimodule,
\item[-]  $p$ is a morphism of unital commutative algebras,
\item[-]  $M, A_n,\dots, A_3$ are  $C$-bimodules,
\item[-]  $j,d_{n-1},\dots, d_2$ are  $C$-equivariant.
\end{itemize}
\end{Definition}
\noindent According to Remark \ref{rk:acute}, the last condition on $d_2$ means that $$d_2(c*a_3)=c*d_2(a_3)\,,\qquad d_2(a_3*c)=d_2(a_3)*c\,,$$ for $c\in C$ and $a_3\in A_3$. Notice that, right sides of both equations make sense, since $C$ acts on $d_2(A_3)=\Ker(\partial)$.
In the case of crossed biextensions, $d_2=j$, and the condition is trivially verified.

An obvious generalization of the previous discussion leads to the definition of the category $\XBiext^n$ and to the following result.

\begin{Theorem} \label{thm:fib_opfib_AssAlg_n}
The diagram below is a fiberwise opfibration in $\Fib(\AssAlg_1)$:
\begin{equation}\label{diag:down_n}
\begin{aligned}
\xymatrix{\XBiext^n\ar[rr]^\Pi\ar[dr]_{\Pi_0}&&\Bimod\ar[dl]^{(\underline{\ })_0}\\&\AssAlg_1}
\end{aligned}
\end{equation}
where the functor $\Pi_0$ and $(\underline{\ })_0$ are the restrictions of the corresponding ones as in Theorem \ref{thm:fib_opfib_Moore}, with $\cC=\AssAlg$.
\end{Theorem}
\begin{Remark}{\em
In \cite[Theorem 4.3]{BauMi02}, the authors proved that, fixed a bimodule $M$ over a unital associative algebra $C$, similarity classes of crossed $n$-fold biextensions over the $C$-bimodule $M$ form an abelian group isomorphic to the $(n+1)$-th Hochschild cohomology group $HH^{n+1}(C,M)$. So also Hochschild cohomology fits in our fibrational framework, namely such groups (or better their underlying sets) are recovered as the fibers of $\bar\Pi$ in the comprehensive factorization:
$$
\begin{aligned}
\xymatrix{
\XBiext^n\ar[r]^Q\ar[dr]_{\Pi_0}
&\overline{\XBiext}^n\ar[r]\ar[d]_{\bar\Pi_0}\ar[r]^{\bar\Pi}
&\Bimod\ar[dl]^{(\underline{\ })_0}\\&\AssAlg_1}
\end{aligned}
$$
}\end{Remark}


\appendix

\section{Basic 2-categorical notions and results} \label{sec:appendixA}

For the reader's convenience, this section provides some 2-categorical notions and results that we use throughout the paper.

Let $\cK$ be a 2-category. Recall from \cite{Str74} that a \emph{comma object} over the opspan in $\cK$:
$$
\xymatrix{A\ar[r]^-f&B&C\ar[l]_-g}
$$
is a four-tuple $((f\downarrow g),p_0,p_1,\lambda)$ in $\cK$
$$
\xymatrix{(f\downarrow g)\ar[r]^-{p_1}\ar[d]_{p_0}
&C\ar[d]^g
\\
A\ar[r]_f
\ar@{}[ur]|(.3){}="1"\ar@{}[ur]|(.7){}="2"\ar@{=>}"1";"2"_\lambda
&B}
$$
that establishes an isomorphism, 2-natural in the first variable
\begin{equation}\label{eq:universal_prop_comma}
\cK(X, (f\downarrow g) ) \cong (\cK(X, f)\downarrow\cK(X,g))
\end{equation}
where the expression on the right is the usual comma category of the functors $\cK(X, f)$ and $\cK(X,g)$. We will refer to this as to the universal property of comma objects.
Notice that sometimes
we shall write just  $(f\downarrow g)$ to mean the comma object of $f$ and $g$. If $g=1_B$, we simplify notation in $(f\downarrow B)$; accordingly, if also $f=1_B$, we write $(B\downarrow B)$.

Comma objects of the identities are rather important
\begin{equation}\label{diag:(B|B)}
\begin{aligned}
\xymatrix{(B\downarrow B)\ar[r]^-{d_1}\ar[d]_{d_0}
&B\ar[d]^1
\\
B\ar[r]_1
\ar@{}[ur]|(.3){}="1"\ar@{}[ur]|(.7){}="2"\ar@{=>}"1";"2"_\epsilon
&B}
\end{aligned}
\end{equation}
since they make the 2-cells of $\cK$ representable in terms of the 1-cells. Indeed, it is easy to see that, given a 2-cell $\alpha\colon f\Rightarrow g\colon A\to B$, $\alpha$ determines by the bijection (\ref{eq:universal_prop_comma}) a unique 1-cell $\underline\alpha\colon A\to (B\downarrow B)$ such that $\epsilon\cdot \underline\alpha=\alpha$. We call $\epsilon$ the universal 2-cell of $B$, and $i$ the diagonal $B\to (B\downarrow B)$ determined by the identity square of $1_B$.

Let us recall from \cite{Str74} that $\cK$ is said to be \emph{representable} if it has comma-objects of identities and $2$-pullbacks, so that $\cK$ has all comma objects. All the 2-categories considered in the present work are representable and finitely complete, namely $\CAT$, $\CAT/\cA$, $\Fib(\cA)$ and $\opFib(\cA)$. Hence, from now on, we shall suppose that $\cK$ is representable and finitely complete.

A 1-cell $f\colon A\to B$ is final (initial) if it is left orthogonal to any discrete (op)fibration (see Proposition \ref{prop:Chevalley}). Explicitly this means that for every commutative square
$$
\xymatrix{
A\ar[d]_f\ar[r]^{h}
&C\ar[d]^{g}
\\
B\ar@{-->}^d[ur]\ar[r]_k
&D}
$$
with $g$ a discrete (op)fibration, there exists a unique diagonal $d\colon B\to C$ such that $d\cdot f=h$ and $g\cdot d=k$.

\smallskip
\emph{Identees} and \emph{coidentifiers} are the 2-dimensional analogs of kernels and cokernels that we need to describe the factorization in  (\ref{diag:factorization_of_p}).

Following \cite{Street_size}, identees are the 2-categorical limits defined as follows: the identee of a 1-cell $f\colon A\to B$ is a 2-universal 2-cell $\kappa$ such that $f\cdot \kappa=id$.
$$
\xymatrix@C=10ex{
I(f)\ar@/^3ex/[r]^d_{}="1"\ar@/_3ex/[r]_c^{}="2" \ar@{=>}"1";"2"^\kappa&A\ar[r]^f&B}
$$
If $\cK$ is finitely complete, identees can be computed by means of the following pullback:
\begin{equation}\label{diag:identee}
\begin{aligned}
\xymatrix{
I(f)\ar[d]\ar[r]^-{\underline \kappa}
&(A\downarrow A)\ar[d]^{(f\downarrow f)}
\\
B\ar[r]_-{i}
&(B\downarrow B)
}
\end{aligned}
\end{equation}
from which we obtain $\kappa$ by composing $\underline \kappa$ with the universal 2-cell of $A$. Moreover, since $i$ is a mono, we can actually consider the identee $I(f)$  as a subobject of $(A\downarrow A)$.

Coidentifiers (or more precisely, the dual notion of identifier) have been introduced in \cite{Kelly89}. They are 2-categorical colimits defined as follows. The coidentifier of a 2-cell $\varphi\colon f\Rightarrow g\colon A\to B$, is a 2-universal 1-cell, $j\colon B\to J(\varphi)$ such that $j\cdot \varphi=id$.
$$
\xymatrix@C=10ex{
A\ar@/^3ex/[r]^f_{}="1"\ar@/_3ex/[r]_g^{}="2" \ar@{=>}"1";"2"^\varphi&B\ar[r]^-j&J(\varphi)}
$$
Discrete (op)fibrations between categories reflect equalities. The same can be said for internal discrete (op)fibrations in $\cK$. This is proved in the following statement.
\begin{Lemma}\label{lm:disc=disc}
Let $g\colon C\to D$ be a discrete (op)fibration, and $\alpha\colon  d \Rightarrow  c\colon X\to C$ a 2-cell such that $g\cdot \alpha=id_{g\cdot d}$. Then $\alpha=id_{d}$.
\end{Lemma}
\begin{proof}
We prove the statement for discrete opfibrations only, since in case of fibrations the proof is similar.

Let  $\underline\alpha$ be the 1-cell canonically determined by the 2-cell $\alpha$. The hypotesis  $g\cdot \alpha=id_{g\cdot d}$ amounts to the fact that the arrow $(g\downarrow g)\cdot \underline\alpha$ factors through the diagonal of $(D\downarrow D)$. This is described by the following commutative diagram of solid arrows
$$
\xymatrix@C=12ex{
X\ar[r]^d\ar[dr]_{\underline\alpha}
&C\ar@{-->}[d]^{\Delta_A}\ar[r]^{g}
&D\ar[d]^{\Delta_D}
\\
&(C\downarrow C)\ar[r]_{(g\downarrow g)}
&(D\downarrow D)
}
$$
where the $\Delta$'s are diagonals.
Then the  thesis is  $\underline\alpha=\Delta_A\cdot d$. We can easily prove it by composing with the projections of the following pullback, which is granted by Proposition \ref{prop:Chevalley}.
$$
\xymatrix@C=12ex{
(C\downarrow C)\ar[r]^-{(g\downarrow g)}\ar[d]_{d_0}
&(D\downarrow D)\ar[d]^{d_0}
\\
C\ar[r]_-{g}
&D
}
$$

\end{proof}

Next proposition has an elementary proof, but it is fundamental for our work. Compare with \cite[Proposition 3.6]{Str2010}.

\begin{Proposition}\label{prop:coid_final_initial}
Coidentifiers are final and initial.
\end{Proposition}
\begin{proof}
Consider the following diagram
$$
\xymatrix{
X\ar@/_2.5ex/[d]_{d}^{}="1"
\ar@/^2.5ex/[d]^{c}_{}="2"
\ar@{=>}"1";"2"^{\alpha}
\\
A\ar[d]_j\ar[r]^{h}
&C\ar[d]^{g}
\\
J(\alpha)\ar@{-->}^d[ur]\ar[r]_k
&D}
$$
where $J(\alpha)$ is a coidentifier of $\alpha$, $g$ is a discrete (op)fibration and $g\cdot h=k \cdot j$. Then of course $g\cdot h\cdot \alpha=k \cdot j\cdot \alpha$, but the term on the right is an identity, since $j$ coidentifies $\alpha$. Therefore, by the previous lemma, $h\cdot \alpha=id$, so that, for the universal property of coidentifiers, there exists a unique $d$ as above such that $j\cdot d=h$. Moreover, by uniqueness of comparisons, $g\cdot d=k$, and this completes the proof.
\end{proof}

\begin{Remark}\label{rk:factorization}{\em
Let $\cK$ be a representable, finitely complete 2-category in which every 2-cell has a coidentifier. Then every 1-cell  $f\colon A\to B$ can be factored as follows: first take the identee $(I(f),\kappa)$ of $f$, then compute its coidentifier $(J(\kappa),j)$. Clearly, since $f\cdot \kappa=id$, $f$ factors to a 1-cell $d$, so that $f=d\cdot j$ is the expected factorization:
$$
\xymatrix@C=8ex{
I(f)\ar@/^3ex/[r]^d_{}="1"\ar@/_3ex/[r]_c^{}="2" \ar@{=>}"1";"2"^\kappa&
A\ar[r]^-j
&J(\kappa)\ar[r]^-d
&B}
$$
}\end{Remark}

\section*{Acknowledgments}

The research that led to the present paper was partially supported by a grant of
the group GNSAGA of INdAM.



\begin{thebibliography}{}

\bibitem{Barr} M. Barr, P. A. Grillet and D. H. van Osdol, Exact categories  and Categories of Sheaves. Lecture Notes in Mathematics, Vol. 236, Springer, Berlin  (1971).


\bibitem{BauMi02} H.-J. Baues and E. G. Minian,  Crossed extensions of algebras and Hochschild cohomology. The Roos Festschrift volume, 1, Homology Homotopy Appl. 4 (2002) 53--82.

\bibitem{BauPi06} H.-J. Baues and T. Pirashvili, Comparison of MacLane, Shukla and Hochschild cohomologies, J. Reine Angew. Math. 598 (2006) 25--69.

\bibitem{prof} J. B\'enabou, Les distributeurs,  Universit\'e catholique de Louvain, Institut de Math\'ematique Pure et Appliqu\'ee, rapport 33 (1973).

\bibitem{Ben89} J. B\'enabou, Some remarks on 2-categorical algebra, Bull. Soc. Math. Belg. S\'er. A 41 (1989) 127--194.

\bibitem{Handbook1}  F. Borceux, Handbook of categorical algebra, Vol. 1,  Cambridge University Press (1994).

\bibitem{BB} F. Borceux and D. Bourn, Mal'cev, Protomodular, Homological and Semi-abelian Categories. Kluwer Academic Publishers (2004).

\bibitem{BJK} F. Borceux, G. Janelidze and G. M. Kelly, Internal object actions, Comment. Math. Univ. Carolin. 46 (2005) 235--255.

\bibitem{bourn} D. Bourn, The tower of $n$-groupoids and the long cohomology sequence, J. Pure Appl. Algebra 62 (1989) 137--183.

\bibitem{proto} D. Bourn, Normalization equivalence, kernel equivalence and affine categories, Springer LNM 1488 (1991) 43--62.

\bibitem{B2000} D. Bourn, Normal functors and strong protomodularity, Theory Appl. Categ. 7 (2000) 206--218.

\bibitem{Bourn02} D. Bourn, Aspherical abelian groupoids and their directions, J. Pure Appl. Algebra 168 (2002) 133--146.

\bibitem{BJ} D. Bourn and G. Janelidze, Protomodularity, descent, and semidirect products, Theory Appl. Categ. 4 (1998) 37--46.

\bibitem{BournJanelidze}  D. Bourn and  G. Janelidze, Extensions with abelian kernel in protomodular categories, Georgian Math. J. 11 (2004) 645--654.

\bibitem{BournPenon} D. Bourn and J. Penon. 2-cat\'egories r\'eductibles. Preprint, University of Amiens Department of Mathematics (1978). Reprinted as TAC Reprints no. 19 (2010).

\bibitem{CasasETC} J. M. Casas, N. Inassaridze, E. Khmaladze and M. Ladra, Adjunction between crossed modules of groups and algebras, J. Homotopy Relat. Struct. 9 (2014) 223--237.

\bibitem{pf} A. S. Cigoli, S. Mantovani and G. Metere, A Push Forward Construction and the Comprehensive Factorization for Internal Crossed Modules, Appl. Categ. Structures 22 (2014) 931--960.

\bibitem{Schreier} A. S. Cigoli, S. Mantovani, G. Metere and E. M. Vitale, General Schreier theory, \emph{in preparation}.

\bibitem{CC} A. S. Cigoli, S. Mantovani and G. Metere, A note on the Chevalley criterion for internal fibrations, \emph{in preparation}.

\bibitem{DedLue66} P. Dedecker, and A. S.-T. Lue,  A nonabelian two-dimensional cohomology for associative algebras, Bull. Amer. Math. Soc. 72 (1966) 1044--1050.

\bibitem{Ger} M. Gerstenhaber, On the deformation of rings and algebras, Ann. Math. Second Series 79 (1964) 59--103.

\bibitem{Gray66} J. W. Gray, Fibred and cofibred categories, in: S. Eilenberg (Ed.), Proc. Conf. on Categorical Algebra, Springer, Berlin  (1966).

\bibitem{He99} C. Hermida, Some properties of Fib as a fibred 2-category, J. Pure Appl. Algebra 134 (1999) 83--109.

\bibitem{Hol} D. Holt, An interpretation of the cohomology groups $H^n(G, M)$, J. Algebra 60 (1979) 307--318.

\bibitem{Hue} J. Huebschmann,  Crossed $n$-fold extensions of groups and cohomology, Comment. Math. Helv. 55 (1980) 302--313.

\bibitem{J03} G. Janelidze, Internal crossed modules, Georgian Math. J. 10 (2003) 99--114.

\bibitem{JMT} G. Janelidze, L. M\'arki and W. Tholen, Semi-abelian categories, J. Pure Appl. Algebra 168 (2002) 367--386.

\bibitem{Kelly89} G. M. Kelly, Elementary observations on 2-categorical limits, Bull. Austr. Math. Soc. 39 (1989) 301--317.

\bibitem{Lue68} A. S.-T. Lue, Non-abelian cohomology of associative algebras, Quart. J. Math. Oxford Ser. 19 (1968) 159--180.

\bibitem{ML77} S. Mac Lane, Historical Note, J. Algebra 60 (1979) 319--320.

\bibitem{Homology} S. Mac Lane, Homology, Springer (1963).

\bibitem{MLW} S. Mac Lane and J. H. C. Whitehead, On the 3-type of a complex, Proc. Nat. Acad. Sci. U. S. A. 36 (1950) 41--48.

\bibitem{MM2010} S. Mantovani and G. Metere, Internal crossed modules and Peiffer condition, Theory Appl. Categ. 23 (2010) 113--135.

\bibitem{mf_vdl} N. Martins-Ferreira and T. Van der Linden, A note on the ``Smith is Huq'' condition, Appl. Categ. Structures 20 (2012) 175--187.

\bibitem{Met2017} G. Metere, A note on strong protomodularity, actions and quotients, J. Pure Appl. Algebra 221 (2017) 75--88.

\bibitem{Behrang} B. Noohi, On weak maps between 2-groups, \emph{preprint}  arXiv:math/0506313 [math.CT] (2005).

\bibitem{Orzech} G. Orzech,  Obstruction theory in algebraic categories I and II, J. Pure Appl. Algebra 2 (1972) 287--314 and 315--340.

\bibitem{Porter} T. Porter, Extensions, crossed modules and internal categories in categories of groups with operations, Proc. Edinburgh Math. Soc. 30 (1987) 373--381.

\bibitem{Rodelo} D. Rodelo, Directions for the Long Exact Cohomology Sequence in Moore Categories, Appl. Categ. Structures 17 (2009) 387--418.

\bibitem{Str74} R. Street, Fibrations and Yoneda's lemma in a 2-category, in: A. Dold, B. Eckmann (Eds.), Category Seminar, Lecture Notes in Mathematics, vol. 420, Springer, Berlin, 1974, pp. 104--133


\bibitem{Str80} R. Street, Cosmoi of internal categories, Trans. Amer. Math. Soc. 258 (1980) 271--318.

\bibitem{Str2010} R. Street and D. Verity, The comprehensive factorization and torsors, Theory Appl. Categ. 23 (2010) 42--75.

\bibitem{Street_size} R. Street, Categories in categories, and size matters, \emph{unpublished, available on the author's webpage}.

\bibitem{Weber} M. Weber, Yoneda structures from 2-toposes, Appl. Categ. Structures 15 (2007) 259--323.

\bibitem{Whi} J. H. C. Whitehead, Combinatorial homotopy II, Bull. Amer. Math. Soc. 55 (1949) 453--496.

\bibitem{Y60} N. Yoneda, On Ext and exact sequences. J. Fac. Sci. Univ. Tokyo Sect. I 8 (1960) 507--576.

\end{thebibliography}
\end{document}